\documentclass[reqno]{amsart}
\usepackage{amssymb,amsmath}
\usepackage{amsthm}
\usepackage{color,graphicx}
\usepackage{xcolor}
\usepackage{hyperref}
\usepackage{verbatim}

\newcommand{\pphi}{\|\partial_\xi^2 \hat{\phi}\|_{L^\infty_{\xi \eta}}}
\newcommand{\ha}{\hat{\phi}}
\newcommand{\va}{\varphi}
\newcommand{\les}{\lesssim}
\newcommand{\si}{\sgn(\xi)}

\newcommand{\lan}{\langle \xi \rangle}
\newcommand{\La}{\|\phi\|_{L^2_{r_1,r_2}}}
\newcommand{\lanx}{\langle x \rangle}
\newcommand{\lany}{\langle y \rangle}
\newcommand{\ti}{\tilde {\chi}}   
\newcommand{\li}{L^\infty_{\xi \eta}}

\newcommand{\D}{D^{a}}

\newcommand{\Dta}{\mathcal{D}^{\alpha}_\xi}

\newcommand{\h}{\mathcal H}

\newcommand{\la}{\langle x \rangle_{N}}
 
\newcommand{\Dt}{\mathcal{D}^{\theta}_\xi}
\newcommand{\dt}{D^{\alpha}_\xi}
\newcommand{\dte}{D^{\theta}_\xi}                                      
\newcommand{\R}{\mathbb R}

\newcommand{\Z}{\mathbb Z}

\newcommand{\p}{\partial}

\newcommand{\sgn}{\text{sgn}}

\numberwithin{equation}{section}

\newtheorem{theorem}{Theorem}[section]
\newtheorem{proposition}[theorem]{Proposition}

\newtheorem{lemma}[theorem]{Lemma}

\newtheorem{claim}[theorem]{Claim}

\begin{document}
\vglue-1cm \hskip1cm
\title[The gBO-ZK equation in weighted spaces]{Persistence properties for the dispersion generalized BO-ZK equation in weighted anisotropic Sobolev spaces}



\author[A. Cunha]{Alysson Cunha}
\address{Instituto de Matem\'atica e Estat\'istica(IME).
Universidade Federal de Goi\'as(UFG), Campus Samambaia, 131, 74001-970, Goi\^ania, Bra\-zil}
\email{alysson@ufg.br}

\author[A. Pastor]{Ademir Pastor}
\address{IMECC-UNICAMP, Rua S\'ergio Buarque de Holanda, 651, 13083-859, Cam\-pi\-nas-SP, Bra\-zil}
\email{apastor@ime.unicamp.br}



\begin{abstract}
In this paper we study  the initial-value problem associated with the dispersion
generalized-Benjamin-Ono-Zakharov-Kuznetsov equation,
$$
u_{t}+D^{a+1}_x \partial_{x}u+u_{xyy}+uu_{x}=0, \qquad a\in(0,1).
$$
 More specifically, we study the persistence property of the solution  in the weighted anisotropic Sobolev spaces
$$
H^{(1+a)s,2s}(\R^{2})\cap L^{2}((x^{2r_1} +y^{2r_2})dxdy),
$$
for appropriate $s$, $r_1$ and $r_2$. By establishing unique continuation properties we also show that our results are sharp with respect to the decay in the $x$-direction.
\end{abstract}
 
\maketitle

\section{Introduction}\label{introduction}

This paper is concerned with the initial-value problem (IVP) associated with the two-dimensional dispersion
generalized-Benjamin-Ono-Zakharov-Kuznetsov (gBO-ZK) equation,
\begin{equation}\label{gbozk}
\begin{cases}
u_{t}+D^{a+1}_x \partial_{x}u+u_{xyy}+uu_{x}=0, \;\;(x,y)\in\R^2, \;t>0, \quad a \in [0,1], \\
u(x,y,0)=\phi(x,y),
\end{cases}
\end{equation}
where $D^{a+1}_x$ stands for the fractional derivative of order $a+1$ with respect to the variable $x$ and is defined, via Fourier transform, as $D^{a+1}_x f(x,y)=(|\xi|^{a+1}\widehat{f})^\vee(x,y)$.

In the limiting case $a=1$, equation in \eqref{gbozk} becomes the Zakharov-Kuznetsov (ZK) equation
\begin{equation}\label{zk}
u_{t}+\partial_{x} \Delta u+uu_{x}=0, 
\end{equation}
while for $a=0$ it reduces to the Benjamin-Ono-Zakharov-Kuznetsov (BO-ZK) equation
\begin{equation}\label{bozk}
u_{t}+\mathcal{H}\partial_x^2 u+u_{xyy}+uu_{x}=0, 
\end{equation}
 where $\h$ denotes the Hilbert transform in the $x$-variable. Equations in \eqref{zk} and \eqref{bozk} appear in physical application. Indeed, the ZK equation was first derived in \cite{ZK} and it models the propagation of nonlinear ion-acoustic waves in magnetized  plasma (see also \cite{LLS} for a rigorous derivation in the long-wave limit of the Euler-Poisson system). On the other hand, the BO-ZK equation was introduced in \cite{Jorge} and \cite{Latorre} and it has
 applications to thin nanoconductors on a dielectric substrate.
 
 From the mathematical viewpoint, equation in \eqref{gbozk} may be seen as a two-dimensional extension of the dispersion generalized Benjamin-Ono equation,
\begin{equation}\label{dgbo}
u_{t}+D^{a+1}_x \partial_{x}u+uu_{x}=0,
\end{equation}
 in much the same way  ZK and BO-ZK equations may be seen as two-dimensional versions of the well-known Korteweg-de Vries and Benjamin-Ono equations, respectively.
 
 Both ZK and BO-ZK equations have been extensively studied in the last two decades. In the next paragraphs we recall some results concerning the well-posedness in weighted Sobolev spaces and which are close to the main issue of this manuscript. Here and throughout the paper by well-posedness we mean in Kato's sense, that is, it includes existence, uniqueness, persistence (if the initial data belongs to some function space $X$ then there exists a unique solution that also belong to $X$) and continuous dependence upon the initial data. In addition, if these properties hold in a small time interval we say the IVP is locally well-posed; on the other, if the properties hold for all $t>0$ we say that the IVP is globally well-posed.  Concerning the ZK equation, the IVP in weighted spaces was studied in \cite{BUM} and \cite{FP}.  In  \cite{BUM} the authors proved the local well-posedness in the isotropic space $H^s(\R^2)\cap L^2((1+x^2+y^2)^{s/2}dxdy)$, $s>3/4$; in their proof they took the advantage of change of variables introduced in \cite{GHerr} in order to explore the symmetric form of \eqref{zk}. On the other hand, in \cite{FP} the authors proved the local well-posedness in the anisotropic spaces $H^s(\R^2)\cap L^2((1+|x|^{2r_1}+|y|^{2r_2})dxdy)$, where $s>3/4$ and $r_1,r_2>0$ are such that $\max\{r_1,r_2\}\leq s/2$. Their proof is a little bit different from the one \cite{BUM}; the main tool is a commutator estimate between weights and the linear group associated with \eqref{zk}. In addition, their method also extend to the generalized nonlinearity $u^ku_x$, $k\geq2$.
 
Concerning the BO-ZK equation, local well posedness in weighted spaces was studied in \cite{AP} from several viewpoints. First the authors proved local well-posedness in $H^s(\R^2)\cap L^2(w^2dxdy)$, $s>2$, provided $w=w(x,y)$ is a weight with bounded derivatives up to order three. In addition, if $r\in(1,5/2)$ and $s\geq2r$ then local well-posedness holds in $\mathcal{Z}_{s,r}:=H^s(\R^2)\cap L^2((1+x^2+y^2)^rdxdy)$. Also, if $r\in[5/2,7/2)$ then local well-posedness in $\mathcal{Z}_{s,r}$ holds provided the initial data $\phi$ is such that $\widehat{\phi}(0,\eta)=0$, for any $\eta\in\R$, where the hat stands for the Fourier transform; in this case, as long as the solution exists it also satisfies $\widehat{u}(0,\eta,t)=0$. These results were shown to be sharp in the sense that a sufficiently smooth nontrivial solution do not persist in $ L^2((1+x^2+y^2)^{7/2}dxdy)$. For recent results concerning local well-posedness in the standard  Sobolev spaces we refer the reader to \cite{APlow} and \cite{Nascimento}.

Another model that extends \eqref{dgbo} to a two-dimensional model is the so-called fractional Zakharov-Kuznetsov equation
\begin{equation}\label{fzk}
u_{t}+D^{a+1} \partial_{x}u+uu_{x}=0,  \quad a\in[0,1],
\end{equation}
where now $D^{a+1}$ is the operator defined  in Fourier variables as $\widehat{D^{a+1}f}(\xi,\eta)=(\xi^2+\eta^2)^{(a+1)/2}\widehat{f}(\xi,\eta)$, which has been studied very recently. By using the short-time Strichartz method introduced in \cite{KT} to deal with the Benjamin-Ono equation the authors in \cite{HLRKW} considered $a=0$ and established local well-posedness in $H^s(\R^2)$, $s>5/3$. They also proved an ill-posedness result in the sense that the data-to-solution map cannot be $C^2$-differentiable from $H^s(\R^2)$ to $H^s(\R^2)$, for any $s\in\R$. The local well-posedness was extended to $0\leq a<1$ in \cite{Schippa} where, by using transversality and localization of time to small frequency dependent time intervals, the author showed the local well-posedness in $H^s(\R^2)$, $s>3/2-a$. In weighted spaces, local well-posedness was studied in \cite{riano} only for $a=0$. In particular it was shown that local well-posedness in $\mathcal{Z}_{s,r}$ holds for $s\geq r$ and $r\in[0,3)$ (with $s>5/3$); if $r\in[3,4)$ then local well-posedness in the same space holds provide the initial data also satisfies $\widehat{\phi}(0,0)=0$. These results are sharp in the sense that no nontrivial solutions persist in $\mathcal{Z}_{4,4}$.

 The IVP \eqref{gbozk} in anisotropic Sobolev spaces $ H^{(1+a)s,2s}(\R^2)$ was studied in \cite{ribaud}. For future references we quote their result in  next theorem.

\vskip.2cm

\noindent {\bf Theorem A.}  Let $E^s=H^{(1+a)s,2s}(\R^2)$.

(a)  Assume $a\in[0,1]$ and $s>\frac{2}{a+1}-\frac{3}{4}$. Then \eqref{gbozk} is locally well-posed in $E^s$.

(b) Assume $a\in(3/5,1]$ and $s=1/2$. Then \eqref{gbozk} is globally well-posed in $E^{1/2}$.
\vskip.2cm

To prove part (a) in Theorem A the authors used the method introduced in \cite{IKT}, which combines the energy method with linear and nonlinear estimates in the short-time Bourgain spaces. Part (b) may be proved taking the advantage of the conservation of the quantities
$$
\int_{\R^2} \left(  \left|D_x^{\frac{a+1}{2}}u\right|^2+u_y^2-\frac{1}{3}u^3 \right)dxdy \quad \mbox{and} \quad \int_{\R^2}u^2\,dxdy
$$
to obtain an priori bound for the local solution.

Let us now turn attention to the results in the present paper. Our purpose here is to extend the  well-posedness  results of Theorem A to anisotropic weighted spaces. Thus our main goal is to establish the persistence property in $ L^2((1+|x|^{2r_1}+|y|^{2r_2})dxdy)$ for appropriate $r_1,r_2\geq0$. As we pointed out above, the cases $a=0$ and $a=1$ have already been treated in the literature. So, we will restrict our attention to the case $a\in(0,1)$; to the best of our knowledge this case has not been treated.

Our first result reads as follows (see next section for the definition of the function spaces).

\begin{theorem}\label{anisogbozk}	Let $a\in (0,1)$ and $r_1,r_2\geq0$. Assume
\[
s>\frac{2}{a+1}-\frac{3}{4} \qquad \mbox{and} \qquad s\geq \frac{2r_2}{1+a}.
\] 
The following statements are true.
	\begin{enumerate}
	\item [1)] If $r_1\in [0,1]$ and $s\geq 1$, then the IVP \eqref{gbozk} is locally well-posed in $\mathrm{Z}^{s}_{r_1,r_2}$.
	\item[2)] If $r_1\in(1,2]$ and $s\geq r_1+\frac{1}{1+a}$, then the IVP \eqref{gbozk} is locally well-posed in $\mathrm{Z}^{s}_{r_1,r_2}$.
	
\item [3)] If $r_1 \in (2,5/2+a)$, $s\geq r_1+\frac{1}{1+a}$ and $r_2>2$, then the IVP \eqref{gbozk} is locally well-posed in $\mathrm{Z}^{s}_{r_1,r_2}$.
		
		\item[4)] If $r_1\in [5/2+a,7/2+a)$, $s\geq r_1+\frac{1}{1+a}$ and $r_2>3$, then the IVP \eqref{gbozk} is locally well-posed in $\dot{\mathrm{Z}}_{r_{1},r_{2}}^{s}$.
\end{enumerate}
In addition the time interval where the solution exists is the same as in Theorem A.
\end{theorem}

Our arguments to prove Theorem \ref{anisogbozk} are inspired in the ones presented in \cite{FLP1}, where the authors proved the well-posedness in weighted spaces for the dispersion generalized BO equation \eqref{dgbo}. Since we are dealing with a two-dimensional model,  the arguments do not follow directly from \cite{FLP1} and we need to deal with many additional terms in the necessary estimates. Due to the nonlocal operator $D^{a+1}_x$,  the most difficult part relies on the estimates when the weights are set in the $x$-direction and we need to play with several product and commutator estimates.

Some remarks concerning the assumptions of Theorem \ref{anisogbozk} are in order. First of all, the condition $s>\frac{2}{a+1}-\frac{3}{4}$  appears in order to have the local well-posedness according to Theorem A, in such a way we spend our efforts  to show the persistence property in the weighted space. The conditions $s\geq1$  and $s\geq r_1+\frac{1}{1+a}$ are used to bound the solution in the resolution space $E^s$. Probably the regularity $s$ may be pushed down to $(1+a)s\geq 2\max\{r_1,r_2\}$, which agrees with the case $a=1$ as described above; however, our strategy do not allow us to achieve this index. In addition, since we use Sobolev's embedding in Fourier variables to estimate some terms, this give rise to the assumptions $r_2>2$ in part 3) and $r_2>3$ in part 4).

Note that part 4) in Theorem \ref{anisogbozk} establishes the well-posedness in $\dot{\mathrm{Z}}_{r_{1},r_{2}}^{s}$, which means that the initial data satisfies $\widehat{\phi}(0,\eta)=0$, for any $\eta\in\R$. Next theorem shows that this is a necessary condition to have local well-posedness in the  following sense: if a sufficiently smooth solution has a decay of order $5/2+a$ in the $x$-direction then the initial data satisfies $\widehat{\phi}(0,\eta)=0$, for any $\eta\in\R$.

\begin{theorem}\label{P1}
	Let $u\in C([0,T]; \mathrm{Z}_{r_1,r_2}^{s})$ be a solution of the IVP
	\eqref{gbozk}, where $a \in (0,1)$, $s\geq 2$ and $r_1,r_2>2$.
	If there exist two different times $t_1, t_2 \in [0,T]$ such that $u(t_j)\in \mathrm{Z}_{5/2+a,r_2}^{s}, \ j=1,2, \ $ then $$\hat{u}(0,\eta,t)=0,$$
for any $\eta\in \R$	and any $t\in[0,T]$.
\end{theorem}

Having Theorem \ref{anisogbozk} in hand, a natural question is what happens if $r_1\geq 7/2+a$. Next theorem establishes that a nontrivial sufficiently smooth solution cannot have such a decay in the $x$-direction. In particular local well-posedness is not expected in $\mathrm{Z}^{s}_{r_1,r_2}$, for $r_1\geq 7/2+a$.

\begin{theorem}\label{UCP} Let $u\in C([0,T]; \mathrm{Z}_{r_1,r_2}^{s})$ be a solution of the IVP
	\eqref{gbozk}, where $a \in (0,1)$, $s\geq 4$ and $r_1,r_2>3$. If there exist three different times $t_1, t_2, t_3\in [0,T]$ such
	that $u(t_j)\in \mathrm{Z}_{7/2+a,r_2}^{s}, \ j=1,2,3$, then $$u(t)\equiv 0,$$
	for any $t\in[0,T]$.
\end{theorem}

Statements in Theorems \ref{P1} and \ref{UCP} may be seen as unique continuation principles. The first results in this direction for nonlocal dispersive equations was put forward by R. Iorio in \cite{Iorio}, \cite{Iorio1} and \cite{Iorio2}, where  the author studied the Benjamin-Ono equation in weighed spaces $L^2((1+|x|^{k})dx)$ with $k$ an integer number. Iorio's results were extended to encompass non-integer values of $k$ in  \cite{GermanPonce}. Then, similar results were established for the dispersion generalized Benjamin-Ono equation \eqref{dgbo} in \cite{FLP1}. Our strategy to prove Theorems \ref{P1} and \ref{UCP} are also inspired in \cite{FLP1}.

The paper is organized as follows. In Section \ref{notation} we introduce some notation and give preliminary results. In particular we recall several product and commutator estimates. In Section \ref{localweighted1} we prove Theorem \ref{anisogbozk}; we divide the proof by first proving the persistence with respect to weights in the $y$-direction and then in the $x$-direction. Finally, in Section \ref{uniquep} we prove Theorems \ref{P1} and \ref{UCP}.


\section{Notation and Preliminaries}\label{notation}

Let us first introduce some notation. We use $c$ to denote various positive constants
that may vary line by line; if necessary we use subscript to indicate
dependence on parameters. Given positive numbers $A$ and $B$, we write $A\lesssim B$ to say that $A\leq cB$ for some positive constant $c$. By $\|\cdot\|_{L^p(\R^d)}$ we denote the usual $L^p(\R^d)$ norm. If no confusion is caused we will use $\|\cdot\|_{p}$ instead of $\|\cdot\|_{L^p(\R^d)}$.
For short we denote the $L^2$ norm simply by $\|\cdot\|$. In particular, if
$f=f(x,y)$ then $\|f\|=\|\|f(\cdot,y)\|_{L^2_x}\|_{L^2_y}$, where by
$\|\cdot\|_{L^2_z}$ we mean the $L^2_z$ norm with respect to the variable
$z$.  The scalar product
in $L^2$ will be then represented by $(\cdot,\cdot)$. For any $s\in \R$, $H^s:=H^s(\R^d)$ represents the usual $L^2$-based
Sobolev space endowed with the norm $\|\cdot\|_{H^s}$. The Fourier transform of $f$ is
defined by
$$
\hat{f}(\zeta)=\int_{\R^d}e^{-ix\cdot\zeta}f(x)dx, \qquad \zeta=(\zeta_1,\ldots,\zeta_d)\in\R^d.
$$
Given any complex number $z$ and a function $f$ defined on $\R^d$, let us define the Bessel and Riesz operators, via their Fourier transforms, as follows
$$
\widehat{J^z_{x_i}f}(\zeta)=(1+|\zeta_i|^2)^{z/2}\hat{f}(\zeta), \quad \widehat{D^z_{x_i}f}(\zeta)=|\zeta_i|^{z}\hat{f}(\zeta), 
$$
$$
\widehat{J^zf}(\zeta)=(1+|\zeta|^2)^{z/2}\hat{f}(\zeta), \quad \widehat{D^zf}(\zeta)=|\zeta|^{z}\hat{f}(\zeta).
$$
Given $s_1,s_2\in \R$, the anisotropic Sobolev space
$H^{s_1,s_2}=H^{s_1,s_2}(\R^2)$ is the set of all tempered distributions $f=f(x,y)$
such that
$$
\|f\|^{2}_{H^{s_1,s_2}}:=\|f\|^{2} +
\|J_x^{s_1}f\|^{2}+\|J_y^{s_2}f\|^{2}<\infty.
$$
We also define the Sobolev spaces in $x$- and $y$-directions, $H_x^{s_1}$ and $H_y^{s_2}$, respectively, as being the set of tempered distributions $f$ such
that
$$\|f\|_{H^{s_1}_x}:=\|J_x^{s_1}f\|<\infty \quad \mbox{and} \quad \|f\|_{H^{s_2}_y}:= \|J_y^{s_2}f\|<\infty.$$

Let $r_1,r_2\in \R$. We define $L^2_{r_1,r_2}$ to be the space all functions $f=f(x,y)$ satisfying
$$
\|f\|_{L^2_{r_1,r_2}}^2:= \int_{\R^2}(1+x^{2r_1}+y^{2r_2})|f(x,y)|^2 dxdy<\infty.
$$
Note that $L^2_{r_1,r_2}=L^2_{r_1,0}\cap L^2_{0,r_2}$.
For $s_1,s_2,r_1,r_2\in \R$, we denote
$$
\mathcal{Z}_{r_1,r_2}^{s_1,s_2}:=H^{s_1,s_2}(\R^2)\cap L^2_{r_1,r_2}(\R^2),
$$
 The norm in $\mathcal{Z}_{r_1,r_2}^{s_1,s_2}$ is
given by
$\|\cdot\|_{\mathcal{Z}_{r_1,r_2}^{s_1,s_2}}^2=\|\cdot\|_{H^{s_1,s_2}}^2+\|\cdot\|_{L^2_{r_1,r_2}}^2$.
Also, the subspace $\dot{\mathcal{Z}}_{r_1,r_2}^{s_1,s_2}$ of $\mathcal{Z}_{r_1,r_2}^{s_1,s_2}$ is defined as
$$
\dot{\mathcal{Z}}_{r_1,r_2}^{s_1,s_2}:=\{f\in\mathcal{Z}_{r_1,r_2}^{s_1,s_2}\ | \
\hat{f}(0,\eta)=0, \ \eta\in\R \}.
$$
Finally, the spaces $\mathrm{Z}^{s}_{r_1,r_2}$ and $\dot{\mathrm{Z}}^{s}_{r_1,r_2}$ are defined as
$$
\mathrm{Z}^{s}_{r_1,r_2}:= \mathcal{Z}_{r_{1},r_{2}}^{(1+a)s,2s} \quad \mbox{and} \quad \dot{\mathrm{Z}}^{s}_{r_1,r_2}:= \dot{\mathcal{Z}}_{r_{1},r_{2}}^{(1+a)s,2s}.
$$

Suppose  $\phi\in \mathcal{Z}_{r_1,r_2}^{s_1,s_2}$ and let $u$ be the corresponding local solution of
\eqref{gbozk}. Assuming  that $u$ is sufficiently regular, we can integrate
the equation with respect to $x$ to obtain
\begin{equation}\label{consquan}
\int_{\R}u(x,y,t)dx=\int_\R\phi(x,y)dx, \quad y\in\R
\end{equation}
as long as the solution exists. This implies that
\begin{equation}\label{fourieru}
\hat{u}(0,\eta,t)=\hat{\phi}(0,\eta), \quad \eta\in \R,
\end{equation}
for all $t$ for which the solution exists. In particular, if $\phi\in
\dot{\mathcal{Z}}_{r_1,r_2}^{s_1,s_2}$ then $u(t)\in
\dot{\mathcal{Z}}_{r_1,r_2}^{s_1,s_2}$ for any
$t$ for which the solution exits.

Next, we introduce some preliminaries results which will be useful to
prove our main results. We start with  some commutator estimates. 

\begin{theorem}\label{Comu}
For any $p\in (1,\infty)$ and $l,m\in \Z^{+}\cup \{0\},$ with $l+m\geq 1,$ there exists a constant  $c>0$, depending only on $p,l$, and $m$ such that
\begin{equation*}
\|\partial_{x}^{l}[\mathcal{H};g]\partial_{x}^{m}f\|_{L^p(\R)}\lesssim \|\partial_{x}^{l+m}g\|_{L^\infty(\R)}\|f\|_{L^p(\R)},
\end{equation*}
where $\partial_{x}^k$ denotes the derivative of order $k$.
\end{theorem}
\begin{proof}
This is a generalization of the Calder\'on commutator estimates \cite{Cald}. See Lemma 3.1 in \cite{Dawson} or Theorem 6 in \cite{GermanPonce}.
\end{proof}

\begin{proposition}\label{Jota}
Let  $\varrho \in L^{\infty}(\R)$, with $\partial_x^k \varrho \in L^{2}(\R)$ for $k=1,2$. Then, for any $\theta\in (0,1)$, there exists a constant $c>0$, depending only on $\varrho$ and $\theta$, such that
\begin{equation}\label{Jota1}
\|[J^\theta;\varrho]f\|_{L^2(\R)}\leq c \|f\|_{L^2(\R)}.
\end{equation}
In addition,
\begin{equation}\label{Jotaf1}
\|J^\theta(\varrho f)\|_{L^2(\R)}\leq c \|J^\theta f\|_{L^2(\R)}.
\end{equation}
\end{proposition}
\begin{proof}
See Propositions 2.4 and 2.5 in \cite{FLP1}.
\end{proof}

\begin{proposition}\label{ComuDerivada}
	Let $0\leq \alpha<1$, $0<\beta\leq 1-\alpha$, $1<p<\infty$ and $d\geq 1$, then 
	\begin{equation*}
	\|D^{\alpha}[D^{\beta};g]D^{1-(\alpha+\beta)}f\|_{L^p(\R^d)}\leq c\, \|\nabla g\|_{L^\infty(\R^d)}\|f\|_{L^p(\R^d)},
	\end{equation*}
where $c$ depends on $\alpha,\beta,p$, and $d$.
\end{proposition}
\begin{proof}
	This result is a consequence of Proposition 3.10 in \cite{dong}. For a similar result in the one-dimensional case see Proposition 3.2 in \cite{Dawson}. See also Proposition 2.2 in \cite{FLP1}.
\end{proof}

\begin{proposition}\label{C}
	If $f\in L^{2}(\R)$ and $\Phi \in H^{2}(\R),$ then
\begin{equation*}
\|[D^{\alpha};\Phi]f\|_{L^2(\R)}\lesssim \|\Phi\|_{H^2(\R)}\|f\|_{L^2(\R)},
\end{equation*}
where $\alpha \in (0,1)$.
\end{proposition}
\begin{proof}
See Proposition 2.12 in \cite{dBO}.
\end{proof}

In what follows, $L^{p}_{s}$ denotes the Sobolev space defined as $L^{p}_{s}:=(1-\Delta)^{-s/2}L^{p}(\R^d)$. Such spaces
can be characterized by the Stein derivative of order $b$ as follows.

\begin{theorem}\label{stein}
	Let $b\in (0,1)$ and $2d/(d+2b)<p<\infty.$ Then $f\in L^{p}_{b}(\R^{d})$ if and only if
	\begin{itemize}
		\item [a)] $f\in L^{p}(\R^{d}),$ 
		\item [b)]
		$\mathcal{D}^{b}f(x):={\displaystyle \left (
			\int_{\R^{d}}\frac{|f(x)-f(y)|^{2}}{|x-y|^{d+2b}}dy\right)^{1/2}} \in
		L^{p}(\R^{d}),$ with
		\begin{equation}\label{equiv}
		\|f\|_{b,p}:=\|J^{b}f\|_{p}\simeq \|f\|_{p}+\|D^{b}f\|_{p}\simeq \|f\|_{p}+\|\mathcal{D}^{b}f\|_{p}.
		\end{equation}
	\end{itemize}
\end{theorem}
\begin{proof}
	See Theorem 1 in \cite{Stein}.
\end{proof}

 From the last equivalence in
\eqref{equiv} we see that the $L^p$ norms of $D^b$ and $\mathcal{D}^{b}$ are equivalent. The advantage in using $\mathcal{D}^{b}$	is that it is suitable when dealing with pointwise  estimates, as we will se below. In addition, from  Fubini's theorem we have the product estimate (see \cite[Proposition 1]{NahasPonce})
\begin{equation}\label{Leib}
\|\mathcal{D}^{b}(fg)\|_{L^2(\R^d)} \leq \|f\mathcal{D}^{b}g\|_{L^2(\R^d)} + \|g\mathcal{D}^{b}f\|_{L^2(\R^d)}.
\end{equation}

We also recall the following.

\begin{lemma}\label{Leibnitz}
Let $b\in (0,1)$ and $h$ be a measurable function on $\R$ such that $h,h'\in L^{\infty}(\R)$. Then, for all $x\in \R$
\begin{equation}\label{Lei}
\mathcal{D}^b h(x)\lesssim \|h\|_{L^{\infty}(\R)}+\|h'\|_{L^\infty(\R)}.
\end{equation}
Moreover,
\begin{equation}\label{Leibh}
\|\mathcal{D}^{b}(h f)\|_{L^2(\R)} \leq \|\mathcal {D}^b h\|_{L^\infty(\R)} \|f\|_{L^2(\R)} + \|h\|_{L^\infty(\R)} \|\mathcal{D}^{b}f\|_{L^2(\R)}.
\end{equation}
\end{lemma}
\begin{proof}
For \eqref{Lei} see Lemma 2.7 in  \cite{pastran}. Note that \eqref{Leibh} is a consequence of \eqref{Leib}.
\end{proof}

Some pointwise estimates in terms of the Stein derivative is given below. We start by introducing a cut-off function 
\begin{equation}\label{varphi}
\varphi\in C_0^\infty(\R) \ \mbox{such that}\ \mbox{supp}\ \varphi\subset [-2,2] \ \mbox{ and} \ \varphi\equiv1 \ \mbox{in}  \ (-1,1).
\end{equation}

\begin{proposition}\label{Dstein}
For any $\theta \in (0,1)$ and $\alpha >0,$ the function $\mathcal{D}^\theta (|\xi|^\alpha \varphi (\xi))(\cdot)$ is continuous in $\eta \in \R-\{0\}$  with
$$\mathcal{D}^\theta (|\xi|^\alpha \varphi(\xi))(\eta) \sim \left\{\begin{array}{lcc}
c|\eta|^{\alpha -\theta}+c_1,& \quad \alpha \not= \theta, |\eta|\ll 1, \\
c(-\ln |\eta|)^{1/2}, & \quad \alpha=\theta, |\eta|\ll 1,\\
\frac{c}{|\eta|^{1/2+\theta}}, & \quad  |\eta|\gg 1,
\end{array}\right.
$$
 in particular, one has that
\begin{equation}\label{Dstein4}
\mathcal{D}^\theta (|\xi|^\alpha \varphi (\xi))\in L^{2}(\R) \ \mbox{if and only if} \ \theta< \alpha +1/2.
\end{equation}
In a similar fashion
\begin{equation}\label{Dstein1}
\mathcal{D}^\theta (|\xi|^\alpha \sgn(\xi) \varphi (\xi))\in L^{2}(\R) \ \mbox{if and only if} \ \theta< \alpha +1/2.
\end{equation}
\end{proposition}
\begin{proof}
See Proposition 2.9 in \cite{FLP1}.
\end{proof}

Note that in the above proposition we are always taking $\alpha>0$. However, in the proof of our main results we also need $\alpha<0$. This is the content of the next two results.

\begin{proposition}\label{DsteinL2}
If $\gamma \in [0,1/2)$ then
\begin{equation}\label{Dstein2}
\mathcal{D}^\gamma (|\xi|^{\gamma-1/2}\varphi (\xi))\notin L^{2}(\R),
\end{equation}
where by $\mathcal{D}^0$ we mean the identity operator.
\end{proposition}
\begin{proof}
See Proposition 2.11 in \cite{dBO}.
\end{proof}

\begin{proposition}\label{DsteinL3}
If $\gamma \in (0,1/2)$ and $0<\epsilon<\gamma$ then
\begin{equation}\label{Dstein3}
\mathcal{D}^{\gamma-\epsilon} (|\xi|^{\gamma-1/2}\varphi (\xi))\in L^{2}(\R).
\end{equation}
\end{proposition}
\begin{proof} Here we use the same approach as in the proof of Proposition 2.9 in \cite{FLP1}. By setting $\theta=\gamma-\epsilon$ and $\gamma_1=\gamma-1/2$, we  see that for $\eta \neq 0$, $\mathcal{D}^\theta (|\xi|^{\gamma_1}\varphi(\xi))(\eta)$ is continuous in $\delta<|\eta|<\frac{1}{\delta}$, for all $\delta>0$.
First, we assume $0<\eta<2/3$. Then
	\begin{equation}
	\begin{split}\label{Aeta}
	[\mathcal{D}^\theta (|\xi|^{\gamma_1}\varphi(\xi))(\eta)]^2 &=\int \frac{(|y|^{\gamma_1}\varphi(y) -|\eta|^{\gamma_1}\varphi(\eta))^{2}}{|y-\eta|^{1+2\theta}}dy\\
	&=\int \frac{(|\xi+\eta|^{\gamma_1} \varphi(\xi+\eta)-|\eta|^{\gamma_1}\varphi(\eta))^2}{|\xi|^{1+2\theta}}d\xi\\
	&= \int_0^{\eta/2}+\int_{\eta/2}^\infty+\int_{-\infty}^{-\eta/2}+\int_{-\eta/2}^0\\
	&=: I_1+I_2+I_3+I_4.
	\end{split}
	\end{equation}
	Next, we  deal with the first integral on the right-hand side of \eqref{Aeta}. In view of $0<\eta<\xi+\eta<3\eta/2<1$, it follows that $\varphi(\xi+\eta)=\varphi(\eta)=1$. Hence, by the mean value theorem there exists $z\in(\eta, \xi+\eta)$ such that
	\begin{equation*}
	\begin{split}
	\eta^{\gamma_1}-(\xi+\eta)^{\gamma_1}=-\gamma_1 z^{\gamma_1 -1}\xi
	\les \xi z^{\gamma_1 -1}
	\les \xi \eta^{\gamma_1 -1},
	\end{split}
	\end{equation*}
	where  we used that $\gamma_1<0$ and $\eta<z$. Thus, from \eqref{Aeta}
	\begin{equation*}
	\begin{split}
	I_1 \les \int_0^{\eta/2}\frac{\xi^2 \eta^{2(\gamma_1 -1)}}{|\xi|^{1+2\theta}}d\xi
	\les  \eta^{2(\gamma_1-1)}\int_{0}^{\eta/2}\xi^{1-2\theta}d\xi
	\les  \eta^{2(\gamma_1-1)} (\eta/2)^{2(1-\theta)}
	\les \eta^{2\epsilon-1}.
	\end{split}
	\end{equation*}
Also,
	\begin{equation*}
	\begin{split}
	I_2\leq \int_{\eta/2}^\infty \frac{((\xi+\eta)^{\gamma_1}+\eta^{\gamma_1})^2}{|\xi|^{1+2\theta}}d\xi
	\leq \eta^{2\gamma_1}\int_{\eta/2}^\infty \xi^{-1-2\theta}d\xi
	\les \eta^{2\epsilon-1}.
	\end{split}
	\end{equation*}
	
With respect to $I_4$ we see that $-\eta/2<\xi<0$ implies $\eta/2<\xi+\eta<\eta<2/3$. Using the mean value theorem again we obtain $(\xi+\eta)^{\gamma_1}-\eta^{\gamma_1}\les |\xi|\eta^{\gamma_1 -1}$.
	Thus
	\begin{equation*}
	\begin{split}
	I_4=\int_{-\eta/2}^0 \frac{((\xi+\eta)^{\gamma_1}-\eta^{\gamma_1})^2}{|\xi|^{1+2\theta}}d\xi
	   \les \int_{-\eta/2}^0\frac{\xi^2 \eta^{2(\gamma_1-1)}}{|\xi|^{1+2\theta}}d\xi
	   =\int_0^{\eta/2}\frac{\xi^2 \eta^{2(\gamma_1-1)}}{|\xi|^{1+2\theta}}d\xi
	   \les \eta^{2\epsilon-1}.
     \end{split}
	\end{equation*}
	Concerning $I_3$ we write
	\begin{equation}\label{itt}
	I_3=\int_{-\infty}^{-\eta/2}\frac{(|\xi+\eta|^{\gamma_1}\varphi(\xi+\eta)-\eta^{\gamma_1})^2}{|\xi|^{1+2\theta}}d\xi=\int_{-\infty}^{-2-\eta}+\int_{-2-\eta}^{-\eta/2}=:I_3^1+I_3^2.
    \end{equation}		
In the first integral in \eqref{itt} we have $\varphi(\xi+\eta)=0$. Hence
	\begin{equation*}
	\begin{split}
	I_3^1\leq \int_{-\infty}^{-2-\eta}\frac{\eta^{2\gamma_1}}{|\xi|^{1+2\theta}}d\xi
	=c\,\eta^{2\gamma_1}(2+\eta)^{-2\theta}
	\leq c\, \eta^{2\gamma_1}\eta^{-2\theta}
	=c\,\eta^{2\epsilon-1}.
	\end{split}
	\end{equation*}
	The second integral in \eqref{itt} can be estimated as
	\begin{equation}
	\begin{split}\label{ittt}
	I_3^2\les \int_{-2-\eta}^{-\eta/2}\frac{(\xi+\eta)^{2\gamma_1}}{|\xi|^{1+2\theta}}d\xi+\int_{-2-\eta}^{-\eta/2}\frac{\eta^{2\gamma_1}d\xi}{|\xi|^{1+2\theta}}
	=:I_3^{2,1}+I_3^{2,2}.
	\end{split}
	\end{equation}
Now we have
    \begin{equation*}
	\begin{split}
	I_3^{2,2}=\eta^{2\gamma_1}\int_{\eta/2}^{2+\eta}\xi^{-1-2\theta}d\xi
	=\frac{\eta^{2\gamma_1}}{2\theta}\Big[(\eta/2)^{-2\theta}-(2+\eta)^{-2\theta}\Big]
	\les \eta^{2\epsilon-1}.
	\end{split}
	\end{equation*}
The first integral on the right-hand side of \eqref{ittt} can be decomposed as 
 \begin{equation*}
	\begin{split}
	I_3^{2,1}&=\int_{-2-\eta}^{-\eta}+\int_{-\eta}^{-\eta/2}=:I+\tilde I.
	\end{split}
	\end{equation*}
To estimate  $\tilde I$, by using that $\eta/2\leq |\xi|$ we deduce
\begin{equation*}
	\begin{split}
	\tilde I
	\les \eta^{-1-2\theta}\int_{-\eta}^{-\eta/2}(\xi+\eta)^{2\gamma_1}d\xi
\les 	\eta^{-1-2\theta} \eta^{2\gamma_1+1}
\les \eta^{2\epsilon-1}.
	\end{split}
	\end{equation*}

To deal with the integral $I$ we  choose $p,q$ such that $1<p<-\frac{1}{2\gamma_1}$ and $
\frac{1}{p}+\frac{1}{q}=1$.
Hence, by  Young's inequality we obtain
\begin{equation*}
	\begin{split}
	 I
	&\leq \eta^{-2\theta}\int_{-2-\eta}^{-\eta}\frac{(\xi+\eta)^{2\gamma_1}}{|\xi|}d\xi\\
	&\les \eta^{-2\theta}\left(\int_{-2-\eta}^{-\eta}(\xi+\eta)^{2p\gamma_1}d\xi+\int_{-2-\eta}^{-\eta}\frac{d\xi}{|\xi|^q}\right)\\
&\les \eta^{-2\theta}\big[1+(2+\eta)^{1-q}\big].
	\end{split}
	\end{equation*}
This completes the proof if $0<\eta<2/3$.	The case $-2/3<\eta<0$ may be treated similarly. 
	
	Next we suppose $\eta>200$. Here,
	\begin{equation}
	\begin{split}
	\mathcal{D}^\theta (|\xi|^{\gamma_1}\varphi(\xi))(\eta)^2&=\int \frac{(\xi+\eta)^{2\gamma_1}\varphi(\xi+\eta)^2}{|\xi|^{1+2\theta}}d\xi\\
	&\les \int_{-2-\eta}^{2-\eta}\frac{(\xi+\eta)^{2\gamma_1} d\xi}{|\xi|^{1+2\theta}}\\
	&\les \frac{1}{(\eta-2)^{1+2\theta}}.
	\end{split}
	\end{equation}
	The case $\eta<-200$ may be treated in a similar fashion. The proof of the proposition is thus completed.
\end{proof}

In the next two results we recall some pointwise estimates we need in the sequel.

\begin{lemma}\label{Pontual1}
Let $b\in (0,1)$. For any $t>0$,
\begin{equation*}
\mathcal{D}^{b}(e^{-itx|x|^{1+a}})\les t^{b/(2+a)}+t^{b}|x|^{(1+a)b}.
\end{equation*}
\end{lemma}
\begin{proof}
See Proposition 2.7 in \cite{FLP1}.
\end{proof}

\begin{lemma}\label{P}
Let $b\in (0,1),$ then for all $t>0$ and $\eta\in \R$,
$$
\mathcal{D}^{b}(e^{it\eta^{2}x})\les\eta^{2b}t^{b}.
$$
\end{lemma}
\begin{proof}
	See Lemma 2.9 in \cite{AP}.
\end{proof}



Since we will be dealing with weighted spaces, let us introduce the truncated weights $\langle x\rangle_N$, $N\in \Z^{+}$, by letting
\begin{eqnarray*}
\langle x\rangle_N:=\left\{\begin{array} {lccc}
\langle x \rangle \ \mathrm{if} \  |x|\leq N,\\
2N \ \mathrm{if} \ |x|\geq 3N,
\end{array} \right.
\end{eqnarray*}
where $\langle x \rangle = (1+x^2)^{1/2}$. Also, we assume that $\langle x\rangle_N$
is smooth and non-decreasing in $|x|$ with $\langle x\rangle_N'(x)\leq 1,$ for any
$x\geq 0$, and there exists a constant $c$ independent of $N$ such that
$|\langle x\rangle_N''(x)|\leq c \partial_x^{2}\langle x \rangle.$

\begin{lemma}\label{interx}
Let $\alpha,b>0.$ Assume that $J_{x_i}^{\alpha}f(x_1,x_2)\in L^{2}(\R^2)$ and
$\langle x_j \rangle^b f(x_1,x_2)=(1+x_j^2)^{b/2}f(x_1,x_2)\in L^{2}(\R^2).$ Then, for any
$\beta \in (0,1)$,
\begin{equation}\label{inter1x}
\|J_{x_i}^{\alpha \beta}(\langle x_j \rangle^{(1-\beta)b}f)\|_{L^2_{x_i}}\leq c\|\langle  x_j
\rangle^{b}f\|_{L^2_{x_i}}^{1-\beta}\|J_{x_i}^{\alpha}f\|_{L^2_{x_i}}^{\beta}, \qquad i,j=1,2.
\end{equation}
Moreover,  inequality \eqref{inter1x} is  still valid with $\langle \cdot\rangle_N$
instead of $\langle \cdot \rangle$ with a constant $c$ independent of $N.$
\end{lemma}
\begin{proof}
For the case $i=j$ see Lemma 4 in \cite{NahasPonce}. However, the same proof holds with $i\neq j$.
\end{proof}

To establish some of our estimates in next sections we need the following computations.
Set
\begin{equation}\label{psidef}
 \psi(\xi,\eta,t)=e^{it\xi(\eta^2-|\xi|^{1+a})}.
\end{equation}
Then
\begin{equation}\label{F1}
\partial_{\xi}(\psi\ha)=\psi\big[it(\eta^2 - (2+a)|\xi|^{a+1})\hat{\phi}+\p_\xi \ha \big],
\end{equation}

\begin{equation}
\begin{split}\label{F2}
\p_\xi^{2}(\psi\ha)&= \psi \big[-t(i(2+a)(1+a)\sgn(\xi)|\xi|^a + t(2+a)^2|\xi|^{2(1+a)}-\\
          &\quad - 2t(2+a)|\xi|^{1+a}\eta^2 +t\eta^4)\ha + 2it(\eta^2-(2+a)|\xi|^{1+a})\p_\xi \ha + \p_\xi ^{2}\ha \big]\\
           &=: F_1+\cdot\cdot\cdot + F_7,
\end{split}
\end{equation}
\begin{equation}
\begin{split}\label{F3}
\p_\xi^{3}(\psi\ha)= \ &\psi \Bigg \{\Big[3(2+a)(1+a)\sgn(\xi)t^2\eta^2 |\xi|^a- 3it^3(2+a)^2 \eta^2|\xi|^{2(1+a)}\\
          &+it^3(2+a)^3 |\xi|^{3(1+a)}+ 3it^3(2+a)|\xi|^{1+a}\eta^4-it^3\eta^6-i t a(2+a)(1+a)|\xi|^{a-1}\\
                     & -3t^2(2+a)^2 (1+a)\sgn(\xi)|\xi|^{1+2a}\Big] \ha +\Big[-3it(2+a)(1+a)\sgn(\xi)|\xi|^a \\
           & -3t^2(2+a)^2 |\xi|^{2(1+a)}+6t^2(2+a)\eta^2|\xi|^{1+a}-3t^2\eta^4\Big ]\p_\xi \ha\\
           & +\Big[3it\eta^2 -3it (2+a)|\xi|^{1+a} \Big]\p_\xi^2 \ha+\p_\xi^3 \ha \Bigg\}\\
          =:& \ G_1 + \cdot\cdot\cdot + G_{14},
\end{split}
\end{equation}

\begin{equation}
\begin{split}\label{F5} 
\p_\xi^{4}(\psi\ha)=&\psi \Bigg \{\Big[4a(2+a)(1+a)t^2\eta^2 |\xi|^{a-1}-(7a+3)(1+a)(2+a)^2t^2|\xi|^{2a}+\\
&-9i(1+a)(2+a)^2 t^3\si\eta^2|\xi|^{1+2a}-6i(2+a)^3(1+a)t^3\si|\xi|^{2+3a} +\\
&+6i(2+a)(1+a)t^3\sgn(\xi)\eta^4 |\xi|^a-ita(2+a)(a^2-1)\sgn(\xi)|\xi|^{a-2} \\
& +6t^4(2+a)^2 \eta^4 |\xi|^{2(1+a)}-4(2+a)^3t^4\eta^2|\xi|^{3(1+a)}\\
&+t^4 (2+a)^4 |\xi|^{4(1+a)}-4 t^4(2+a)\eta^6 |\xi|^{1+a}+\\
&+t^4 \eta^8\Big]\ha+\Big[-4ita(2+a)(1+a)|\xi|^{a-1}-12t^2(2+a)^2(1+a)\sgn(\xi)|\xi|^{1+2a}+\\
&+12t^2(1+a)(2+a)\sgn(\xi)\eta^2|\xi|^a +12it^3(2+a)\eta^4 |\xi|^{1+a}+\\
&-12it^3(2+a)^2\eta^2|\xi|^{2(1+a)}+4it^3 (2+a)^3 |\xi|^{3(1+a)}-4it^3 \eta^6\Big]\p_\xi \ha +\\
&+6\Big[-it(2+a)(1+a)\sgn(\xi)|\xi|^a -t^2\eta^4 -t^2 (2+a)^2 |\xi|^{2(1+a)}+2t^2(2+a)\eta^2|\xi|^{1+a}\Big]\p_\xi^2 \ha \\
&+\Big[ 4it \eta^2-4it(2+a)|\xi|^{1+a} \Big]\p_\xi^3 \ha +\p_\xi^4 \ha\Bigg\}\\
=:&H_1+ \cdot\cdot\cdot + H_{25}.
\end{split}
\end{equation}
Note that $F_j$, $G_j$ and $H_j$ depends on $\xi,\eta,t$ and $\hat{\phi}$, that is, $F_j=F_j(\xi,\eta,t,\hat{\phi})$, $G_j=G_j(\xi,\eta,t,\hat{\phi})$ and $H_j=H_j(\xi,\eta,t,\hat{\phi})$.

We end this section with two important estimates that will be used several times in the proof of our main results.

\begin{lemma}\label{DF}Let $\psi$ be as in \eqref{psidef}. For all $\theta \in (0,1)$ and $t\in (0,\infty)$,
\begin{equation*}
\|\Dt(\psi \hat{f})\|\lesssim \rho(t)\Big(\|f\|+\|D_y^{2\theta}f\|+\|D_x^{(1+a)\theta}f\|\Big)+\||x|^\theta f\|,
\end{equation*}
where $\rho(t)=1+t^\theta+t^{\frac{\theta}{2+\theta}}$.
\end{lemma}
\begin{proof}
Using \eqref{Leib} and Lemmas \ref{P} and \ref{Pontual1},
\begin{equation*}
\begin{split}
\|\Dt(\psi \hat{f})\|&\lesssim \|\Dt (e^{it\xi\eta^2})e^{-it\xi|\xi|^{1+a}}\hat{f}\|+\|e^{it\xi\eta^2}\Dt(e^{-it\xi|\xi|^{1+a}}\hat{f})\|\\
&\lesssim t^\theta \|\eta^{2\theta}\hat{f}\|+\|\Dt(e^{-it\xi|\xi|^{1+a}})\hat{f}\|+\|e^{-it\xi|\xi|^{1+a}}\Dt \hat{f}\|\\
&\lesssim t^\theta \|\eta^{2\theta}\hat{f}\|+\|(t^{\frac{\theta}{2+\theta}}+t^\theta |\xi|^{(1+a)\theta})\hat{f}\|+\|\Dt \hat{f}\|\\
&\lesssim \rho(t)\Big(\|\hat{f}\|+\|\eta^{2\theta}\hat{f}\|+\||\xi|^{(1+a)\theta}\hat{f}\|\Big)+\||D_\xi^\theta \hat{f}\|.
\end{split}
\end{equation*}
Then, Plancherel's identity gives us the desired result.
\end{proof}

For the next result we set
\begin{equation}\label{chi}
\chi(\xi,\eta)=\varphi(\xi)\varphi(\eta),
\end{equation}
where $\varphi$ is given by \eqref{varphi}.

\begin{lemma}\label{Dchip} For all $\theta \in (0,1)$, $t\in [0,\infty)$, $\sigma_1\in \{0,1\}$, $\sigma_2\geq 1$ and $\sigma_3\geq 0$, it follows that
\begin{equation}\label{xi}
\|\Dt\big(\chi(\xi,\eta)\psi \sgn(\xi)^{\sigma_1}|\xi|^{\sigma_2}\eta^{\sigma_3} \hat{f}\big)\|\lesssim\|f\|+\||x|^\theta f\|,
\end{equation}
and
\begin{equation}\label{eta}
\|\Dt\big(\chi(\xi,\eta)\psi \eta^{\sigma_3} \hat{f}\big)\|\lesssim\|f\|+\||x|^\theta f\|,
\end{equation}
where the implicit constants depend on $t$ and $a$.
Moreover, if $1/2<a<1$ then
\begin{equation}\label{Da}
\|\Dt\big(\chi(\xi,\eta)\psi \eta^{\sigma_3} |\xi|^a \hat{f}\big)\|\lesssim\|J_x^{2a}f\|+\|\lanx^2 f\|+\|\lany^{\sigma_4} f\|,
\end{equation}
where  $\sigma_4>1$ is an arbitrary number.

This result still holds if we replace $\chi(\xi,\eta)$ by $\tilde \chi(\xi,\eta)=\varphi(\xi)e^{-\eta^2}$.
\end{lemma}
\begin{proof}
We will give the proof of \eqref{xi} with $\sigma_1=1$. The proof of  the other cases are similar. Setting $h(\xi,\eta)=\chi(\xi,\eta)\psi \sgn(\xi)^{\sigma_1}|\xi|^{\sigma_2}\eta^{\sigma_3}$ and noting that $\sgn(\xi)^{\sigma_1}|\xi|^{\sigma_2}=\xi|\xi|^{\sigma_2-1}$, it is easy to see that $h$ together with its derivative with respect to $\xi$ are bounded. So, the result follows as an application of Lemma \ref{Leibnitz}. The proof of \eqref{eta} is similar.

Next we will establish \eqref{Da}. Using \eqref{eta},
\begin{equation}
\begin{split}\label{xi1}
\|\Dt\big(\chi(\xi,\eta)\psi \eta^{\sigma_3} |\xi|^a \hat{f}\big)\|&\lesssim \||\xi|^a \hat f\|+\|D_\xi^\theta(|\xi|^a \hat f)\|\\
&\les \||\xi|^a \hat f\|+\|\p_\xi(|\xi|^a \hat f)\|\\
&\les \||\xi|^a \hat f\|+\||\xi|^{a-1} \hat f\|+\||\xi|^a \p_\xi \hat f\|,
\end{split}
\end{equation}
where  we used the interpolation inequality
$
\|D_\xi^\theta(|\xi|^a \hat f)\|\les_{a,t} \||\xi|^a \hat f\|^{1-\theta}\|\p_\xi(|\xi|^a \hat f)\|^\theta.
$
Now, using Sobolev's embedding 
\begin{equation}
\begin{split}\label{xi2}
\||\xi|^{a-1} \hat f\|&\les \||\xi|^{a-1}\chi \hat f\|+\||\xi|^{a-1}(1-\chi)\hat f\|\\
&\les \||\xi|^{a-1}\chi\|\|\hat f\|_{L^\infty_{\xi\eta}}+\Big\|\frac{1-\chi}{\xi} \Big\|_{L^\infty_{\xi\eta}}\||\xi|^a \hat f\|\\
&\les \|\lanx^2 f\|+\|\lany^{\sigma_4} f\|+\|J_x^{a}f\|.
\end{split}
\end{equation}
In addition, from Lemma \ref{interx} and Plancherel's identity
\begin{equation}\label{xi3}
\||\xi|^a \p_\xi \hat f\|\les \|J_\xi^2 \hat f\|+\|\lan^{2a}\hat f\|\les \|\lanx^2 f\|+\|J_x^{2a}f\|.
\end{equation}
Gathering together \eqref{xi1}--\eqref{xi3} we establish \eqref{Da}.
The proof of the lemma is thus completed.
\end{proof}


\section{Local  well posedness in weighted spaces} \label{localweighted1}

In this section, we prove Theorem \ref{anisogbozk}. So, let us assume that
$\phi\in\mathrm{Z}_{r_1,r_2}^{s}=E^s\cap L^2_{r_1,r_2}$. First of all, we note that the existence of a continuous 
local solution, say $u:[0,T]\to E^s$, is given by Theorem A. 
Thus, we only need to establish the persistence property in $L^2_{r_1,r_2}$. Moreover, once we obtain
the persistence property in $L^2_{r_1,r_2}$, the continuity of $u:[0,T]\to L^2_{r_1,r_2}$ and the
continuity of the map data-solution follow as in \cite[Theorem 1.3]{AP}.

If $r_1=r_2=0$, there is nothing to prove. Hence, we can always assume either $r_1>0$ or $r_2>0$. In addition, recalling that $L^2_{r_1,r_2}=L^2_{r_1,0}\cap L^2_{0,r_2}$ we see that it suffices to prove the persistence in $L^2_{r_1,0}$ and in $L^2_{0,r_2}$.

\vskip.3cm

\noindent {\bf Part 1):}  We will divide in two other cases.

\noindent {\bf Case a).}\noindent {\bf \  Weights in the $y$-direction: persistence in $L^2_{0,r_2}$, $r_2>0$}.

Take  $\phi \in E^s\cap L^2_{0,r_2}$. We multiply  the differential equation
\eqref{gbozk} by $\langle y \rangle_{N}^{2r_2}u$ and integrate on $\R^{2}$ to obtain
\begin{equation}\label{106}
\frac{1}{2}\frac{d}{dt}\|\langle y \rangle_{N}^{r_2}u\|^{2}+
\Big(\langle y \rangle_{N}^{r_2}u,\langle y \rangle_{N}^{r_2}D^{a+1}_x\partial_x u +
\langle y \rangle_{N}^{r_2}u_{xyy}+\langle y \rangle_{N}^{r_2}uu_{x}\Big)=0.
\end{equation}

Let
\begin{equation}\label{Mdef}
M=\sup_{[0,T]}\|u(t)\|_{E^s}.
\end{equation}

Since $\langle y \rangle_{N}$ is independent of $x$ we obtain $\langle y \rangle_{N}^{r_2}D^{a+1}_x\partial_x u=D^{a+1}_x \partial_{x}(\langle y \rangle_{N}^{r_2}u)$. Therefore, taking into account that $D^{a+1}_x\partial_x$ is antisymmetric,
 the contribution of the term $(\langle y \rangle_{N}^{r_2}u,\langle y \rangle_{N}^{r_2}D^{a+1}\partial_x u)$ in \eqref{106} is null. In addition,
 \begin{equation*}
 \big(\langle y \rangle_{N}^{r_2}u,\langle y \rangle_{N}^{r_2}uu_{x}\big)=\frac{1}{3}\int\partial_x(\langle y \rangle_{N}^{2r_2}u^3)=0.
 \end{equation*}

It remains to estimate the middle term in \eqref{106}. To do that, let us first assume $r_2>1/2$.
By Lemma \ref{interx} with $\alpha=2r_2, \
\beta=\frac{1}{2r_2}$, $b=r_2$ and by Young's inequality we see that
\begin{equation}\label{teoZby}
\|J_{y}(\langle y \rangle_{N}^{r_{2}-1/2}u)\|\lesssim \|\langle y \rangle_{N}^{r_2}u\|+\|J^{2r_2}_{y}u\|.
\end{equation}
In a similar fashion,
\begin{equation}\label{teoZbx}
\|J_{x}(\langle y \rangle_{N}^{r_{2}-1/2}u)\|\lesssim \|\langle y \rangle_{N}^{r_2}u\|+\|J_{x}^{2r_2}u\|.
\end{equation}
It is to be clear that to obtain \eqref{teoZby} for instance, we are using Lemma \ref{interx} only in the $y$-direction. In fact, by writing $\|J_{y}(\langle y \rangle_{N}^{r_{2}-1/2}u)\|=\|\|J_{y}(\langle y \rangle_{N}^{r_{2}-1/2}u)\|_{L^2_y}\|_{L^2_x}$, we first use Lemma \ref{interx} in the $y$-direction and then H\"older's inequality in the $x$-variable. An application of Young's inequality then gives \eqref{teoZby}. This kind of argument will be used along the paper without additional comments.

Using integration by parts, the  inequality
$|\partial_{y}\langle y \rangle_{N}^{2r_2}|\lesssim \langle y \rangle_{N}^{2r_{2}-1}$,
\eqref{teoZby} and \eqref{teoZbx} we obtain
\begin{equation}\label{teoZc}
\begin{split}
\int \langle y \rangle_{N}^{2r_2}u\partial_{x}\partial_{y}^{2}u=&\ -\int
\partial_{y}\langle y \rangle_{N}^{2r_2}u\partial_{x}\partial_{y}u  -\underbrace{\int
\langle y \rangle_{N}^{2r_2}\partial_{y}u\partial_{x}\partial_{y}u}_{=0}\\
\lesssim& \
\|\langle y \rangle_{N}^{r_2-1/2}\partial_{x}u\|\|\langle y \rangle_{N}^{r_2-1/2}\partial_{y}u\|\\
 \lesssim &\ \|J_{x}(\langle y \rangle_{N}^{r_2-1/2}u)\|^2 +\|J_{y}(\langle y \rangle_{N}^{r_2-1/2}u)\|^2
+\|\langle y \rangle_{N}^{r_2}u\|^2\\ 
\lesssim & \ \|\langle y \rangle_{N}^{r_2}u\|^2 + \|u\|^2_{H^{2r_2}}\\
\lesssim &\ \|\langle y \rangle_{N}^{r_2}u\|^2+M^2,
\end{split}
\end{equation}
where we used that $E^s\hookrightarrow H^{2r_2}$. 

On the other hand, if $r_2\in(0,1/2]$, we have $|\partial_{y}\langle y \rangle_{N}^{2r_2}|\lesssim \langle y \rangle_{N}^{2r_{2}-1}\lesssim 1$. Hence, as in \eqref{teoZc},
\begin{equation*}
\begin{split}
\int \langle y \rangle_{N}^{2r_2}u\partial_{x}\partial_{y}^{2}u=&\ -\int
\partial_{y}\langle y \rangle_{N}^{2r_2}u\partial_{x}\partial_{y}u  =\int
	\partial_y\langle y \rangle_{N}^{2r_2}\partial_{x}u\partial_{y}u\\
\lesssim & \ 
\|\partial_{x}u\|\|\partial_{y}u\|\lesssim \|u\|^2_{H^1}\lesssim M^2,
\end{split}
\end{equation*}
where now we used that $E^s\hookrightarrow H^1$.
The  implicit constants that appears here and in the rest of the proof  will always be independent of $N$.

From \eqref{106} and the above inequalities  we find that
$$
\frac{d}{dt}\|\langle y \rangle_{N}^{r_2}u\|^{2}\leq c(1+\|\langle y \rangle_{N}^{r_2}u\|^{2}).
$$
So, by the Gronwall lemma (see, for instance, \cite[Theorem 12.3.3]{Hille}), 
$$
\|\langle y \rangle_{N}^{r_2}u\|^{2}\leq \|\langle y \rangle_{N}^{r_2}\phi\|^{2}+tc+c\int_{0}^{t}e^{c(t-t')}(\|\langle y \rangle_{N}^{r_2}\phi\|^{2}+t'c)dt'.
$$
By solving the above integral and using the monotone convergence theorem we get
\begin{equation*}
\|\langle y \rangle^{r_2}u\|^{2}\leq e^{ct}\|\langle y
\rangle^{r_2}\phi\|^{2}+e^{ct}-1.
\end{equation*}
This  proves the persistence property in $L^2_{0,r_2}$.

So in what follows, we only consider weights in the $x$-direction. That is, it remains to show the persistence property in $L^2_{r_1,0}$.\\

\noindent {\bf Case b).} \noindent {\bf Weights in the $x$-direction: persistence in $L^2_{r_1,0}$, $r_1>0$.}\\

Let $r_1\in (0,1].$
Putting $r_1=\theta$, multiplying  
\eqref{gbozk} by $\la^{2\theta}u$ and integrating on $\R^{2}$, we obtain
\begin{equation}\label{106x}
\frac{1}{2}\frac{d}{dt}\|\la^{\theta}u\|^{2}+
\Big(\la^{\theta}u,\la^{\theta}D^{1+a}_x\partial_{x}u+
\la^{\theta}u_{xyy}+\la^{\theta}uu_{x}\Big)=0.
\end{equation}
To start with, following the ideas contained in \cite{FLP1}, we write
\begin{equation}\label{adois0}
\begin{split}
\langle x \rangle_{N}^{\theta}D^{1+a}_x\partial_{x}u=\D_x (\la^{\theta} D_x \partial_{x}u)-[\D_x;\la^{\theta}]D_x\partial_{x}u
=: A_1+A_2.
\end{split}
\end{equation}
From Proposition \ref{ComuDerivada} and the fact that  $\|\partial_x\la^{\theta}\|_{L^\infty_x}\lesssim 1$, we obtain
\begin{equation}\label{adois}
\begin{split}
\|A_2\|=\|[\D_x;\la^{\theta}]D^{1-a}_x D^a_x \partial_{x}u\|
               \lesssim \|\partial_x \la^{\theta}\|_{L^\infty_x}\|\D_x \partial_{x}u\|
               \lesssim \|\D_x \partial_{x}u\|.
\end{split}
\end{equation}
 Inequality \eqref{adois} and the fact that $s\geq1$ yield
\begin{equation*}
\|A_2\|\lesssim \|J_x^{a+1}u\|\lesssim \|J_x^{(a+1)s}u\|\lesssim M,
\end{equation*}
where, as before, $M$ is given in \eqref{Mdef}.
For  $A_1$,  we write
\begin{equation*}
\begin{split}
A_1=\D_x(\la^{\theta} D_x\partial_{x}u)=\D_x \partial_{x}(\la^{\theta} D_x u)-\D_x ((\partial_{x} \la^{\theta})D_x u)
=:\ B_1+B_2.
\end{split}
\end{equation*}
Another application of Proposition \ref{ComuDerivada} together with the fact that $|\partial_x^\alpha\la^\theta|\lesssim 1$, $\alpha=1,2$, yield
\begin{equation*}
\begin{split}
\|B_2\|_{L^2_x}
&\leq \|[\D_x;\partial_{x}\la^\theta]D_x u\|+\|\partial_{x}\la^\theta D^{1+a}_x u\|\\
&=\|[\D_x;\partial_{x}\la^\theta]D^{1-a}_xD^{a}_x u\|+\|\partial_{x}\la^\theta D^{1+a}_x u\|\\
&\lesssim \|\partial_{x}^{2}\la\|_{L^\infty_x} \|\D_x u\|+\|D^{1+a}_x u\|\\
&\lesssim \|J^{1+a}_{x}u\|\lesssim M.
\end{split}
\end{equation*}
Observe that $B_1$ reads as 
\begin{equation*}
\begin{split}
B_1=\D_x \partial_{x}(\la^{\theta} D_x u)=\D_x \partial_{x}D_x(\la^{\theta} u)-\D_x \partial_{x}[D_x;\la^{\theta}]u
                          =:  C_1+C_2.
\end{split}
\end{equation*}
Inserting $C_1$ in \eqref{106x}, from the antisymmetry of operator $D^{1+a}_x\p_x$, we see that its contribution is null. On the other hand,
using that $D_x=\h \partial_x$, we get
\begin{equation*}
\begin{split}
[D_x;\la^{\theta}]u&=D_x(\la^{\theta} u)-\la^{\theta} D_x u\\
       &=\h\partial_{x}(\la^{\theta} u)-\la^{\theta} \h \partial_{x}u\\
       &=\h((\partial_{x}\la^{\theta})u)+[\h;\la^{\theta}]\partial_x u.
\end{split}
\end{equation*}
Therefore,
$$C_2=-\D_x \partial_x \h((\partial_x \la^{\theta})u)-\D_x \partial_x [\h;\la^{\theta}]\partial_x u=:D_1+D_2.$$
From  the interpolation inequality
$\|\D_x u\|_{L^2(\R)}\lesssim \|u\|_{L^2(\R)}^{1-a}\|D_x u\|_{L^2(\R)}^{a}$, Young's inequality, and Theorem \ref{Comu},
we infer
\begin{equation*}
\begin{split}
\|D_2\|&=\|\D_x  \partial_{x}[\h;\la^{\theta}]\partial_x u\|\\
       &\lesssim\|\partial_x [\h;\la^{\theta}]\partial_x u\|^{1-a}\|D_x \partial_x [\h;\la^{\theta}]\partial_x u\|^{a}\\
       &\lesssim\|\partial_x [\h;\la^{\theta}]\partial_x u\|^{1-a}\| \partial_{x}^2 [\h;\la^{\theta}]\partial_x u\|^{a}\\
       &\lesssim\left(\|\partial_{x}^3 \la^{\theta}\|_{\infty}+\|\partial_{x}^2 \la^{\theta}\|_{\infty}\right)\|u\|\\
       &\lesssim \|u\|\lesssim M,
\end{split}
\end{equation*}
where we used that $|\partial_x^\alpha\la^\theta|\lesssim 1$, $\alpha=2,3$. Similarly,
\begin{equation*}
\begin{split}
\|D_1\|&=\|\h \D_x \partial_{x}((\partial_x \la^{\theta}) u)\|\\
       &\lesssim\|\D_x((\partial_{x}^2 \la^{\theta})u)\|+\|\D_x((\partial_x \la^{\theta})\partial_x u)\|\\
       &\lesssim  \|\partial_{x}^2 \la^{\theta} u\|+\|D_x(\partial_{x}^2 \la^{\theta} u)\|+\|[\D_x;\partial_x \la^{\theta}]\partial_x u + \partial_{x}\la^{\theta} \D_x \partial_{x}u\|\\
      & \lesssim  \|\partial_{x}^2 \la^{\theta}\|_{\infty}\|u\|+\|\partial_{x}^3 \la^{\theta}\|_{\infty}\|u\|+\|\partial_{x}^2 \la^{\theta}\|_\infty \|\partial_x u\| +\\
      & \quad + \|[\D_x;\partial_x \la^{\theta}]D^{1-a}_x D^{a}_x u\|+\|\partial_x \la^{\theta}\|_\infty \|\D_x \partial_x u\|\\
      & \lesssim  \left(\|\partial_{x}^2 \la^{\theta}\|_{\infty}+\|\partial_{x}^3 \la^{\theta}\|_{\infty} \right)\|J_x u\|+\| \partial_{x}^2 \la^{\theta}\|_{L^\infty_{x}} \|J_x^{1+a} u\|\\
      & \lesssim  \|J_x^{1+a} u\|\lesssim M.
\end{split}
\end{equation*}
From the above inequalities and \eqref{adois0}, we conclude
\begin{equation}\label{adois2}
\|\langle x \rangle_{N}^{\theta}D^{1+a}_x\partial_{x}u\|\lesssim M.
\end{equation}

Next, using integration by parts and 
$|\partial_{x}\langle x \rangle_{N}^{2\theta}|\les\langle x \rangle_{N}^{2\theta-1}$,
 we obtain
\begin{equation*}
\begin{split}
\int \langle x \rangle_{N}^{2\theta}u\partial_{x}\partial_{y}^{2}u=\ \frac{1}{2}\int
\partial_{x}\langle x \rangle_{N}^{2\theta}(\partial_{y}u)^2
\lesssim
\|\langle x \rangle_{N}^{-1/2+\theta}\partial_{y}u\|^2.
\end{split}
\end{equation*}
If $\theta=r_1\in(0,1/2]$, we promptly see that
\[
\int \langle x \rangle_{N}^{2\theta}u\partial_{x}\partial_{y}^{2}u\lesssim
\|\partial_{y}u\|^2\lesssim\|u\|_{H^1}^2\lesssim M^2.
\]
Also, if $\theta=r_1\in(1/2,1]$, Lemma \ref{interx} and Young's inequality imply
\begin{equation*}
\begin{split}
\int \langle x \rangle_{N}^{2\theta}u\partial_{x}\partial_{y}^{2}u
&\lesssim 
\|\langle x \rangle_{N}^{-1/2+\theta}\partial_{y}u\|^2
\lesssim \|J_y(\langle x \rangle_{N}^{-1/2+\theta}u)\|^2\\
&\lesssim \|\langle x \rangle_{N}^{\theta}u\|+\|J_{y}^{2\theta}u\|
\lesssim \|\langle x \rangle_{N}^{\theta}u\|^2+M^2,
\end{split}
\end{equation*}
where we used that $E^s\hookrightarrow H^{2r_1}_y$. In both cases we get
\begin{equation}\label{adois3}
\int \langle x \rangle_{N}^{2\theta}u\partial_{x}\partial_{y}^{2}u \lesssim \|\langle x \rangle_{N}^{\theta}u\|^2+M^2.
\end{equation}

Finally, since $E^s\hookrightarrow H^{(1+a)s}\hookrightarrow L^\infty$ and $|\partial_{x}\langle x \rangle_{N}^{2\theta}|\lesssim \langle x \rangle_{N}^{2\theta-1}\lesssim \langle x \rangle_{N}^{\theta}$, we deduce
\begin{equation}\label{teoZd12}
\begin{split}
\Big|\left(\langle x \rangle_{N}^{\theta}u,\langle x \rangle_{N}^{\theta}uu_{x}\right)\Big|=\frac{1}{3}\left|\int \partial_x\langle x \rangle_{N}^{2\theta}u^3\right|\lesssim \|\langle x \rangle_{N}^{\theta}u\|\|u\|\|u\|_\infty\lesssim M^4+\|\langle x \rangle_{N}^{\theta}u\|^2
\end{split}
\end{equation}

Combining estimates \eqref{adois2}, \eqref{adois3}, and \eqref{teoZd12} with \eqref{106x}, we deduce
\[
\frac{1}{2}\frac{d}{dt}\|\langle x \rangle_{N}^{\theta}u\|^{2}\leq c(1+\|\langle x \rangle_{N}^{\theta}u\|^2).
\]
By using Gronwall's lemma and arguing as before, we finally obtain
\begin{equation*}
\|\langle x \rangle^{r_1}u\|^{2}\leq e^{2ct}\|\langle x
\rangle^{r_1}\phi\|^{2}+e^{2ct}-1.
\end{equation*}
 This proves Case b) and completes the proof of Part 1).\\

\noindent {\bf Part 2):} The persistence in $L^2_{0,r_2}$ follows exactly as in Part 1). So we need only to prove the persistence in $L^2_{r_1,0}$. Here, instead of using the differential equation itself we will use the  equivalent integral formulation 
\begin{equation}\label{inteq}
u(t)=U(t)\phi -\frac12\int_{0}^{t}U(t-\tau)\partial_{x}u^2(\tau)d\tau,
\end{equation}
where $U(t)\phi$  is  the solution of the IVP associated with the linear
gBO-ZK equation. This is necessary because we are not able to reiterate the process in Part 1). At this point our analysis diverges from that in \cite{FLP1}.

 We will divide into two other cases.

\noindent  \textbf{Case a).} $r_1\in(1,2)$. Let us start by writing $r_1=1+\theta$, $\theta \in (0,1)$. Since
 \begin{equation}\label{r1int}
 \||x|^{r_1} u(t)\|\leq \||x|^{r_1} U(t)\phi\|+\int_0^t \||x|^{r_1} U(t-\tau)z(\tau)\|d\tau, \quad z=\frac12\p_x u^2,
 \end{equation}
 we need to estimate each term on the right-hand side.
 Using \eqref{F1}
 \begin{equation*}
\begin{split}
\||x|^{1+\theta}U(t)\phi\|&\lesssim \|\mathcal{D}_\xi^{1+\theta}\widehat{U(t)\phi}\|=\|\mathcal{D}_\xi^{\theta}\partial_{\xi}(\psi\hat{\phi})\| \\ 
&\lesssim t\Big (\|\mathcal{D}_\xi^\theta(\psi\eta^2 \ha)\|+\|\mathcal{D}_\xi^\theta(\psi|\xi|^{1+a} \ha)\|\Big)+\|\mathcal{D}_\xi^\theta(\psi\p_\xi \ha)\|\\
&=:A_1+A_2+A_3.
\end{split}
\end{equation*}
Now, using Lemma \ref{DF} and Young's inequality,
\begin{equation}\label{aa1}
\begin{split}
A_2 \lesssim t\rho(t)&  \Big(\|D_x^{1+a}\phi\|+\|D^{2\theta}_y D_x^{1+a} \phi\|+\|D_x^{(1+a)\theta}D_x^{1+a} \phi\|\Big)+ \underbrace{\||x|^\theta D_x^{1+a} \phi\|}_{A_{2,1}}\\
\lesssim t\rho(t)& \Big(\|\phi\|_{H_{y}^{2(1+\theta)}}+\|\phi\|_{H_x^{(1+a)(1+\theta)}}\Big)+A_{2,1}.
\end{split}
\end{equation}
Since $E^s\hookrightarrow H_y^{2r_1}$ and $E^s\hookrightarrow H_x^{(1+a)r_1}$ the first two terms in \eqref{aa1} are finite. To estimate $A_{2,1}$ we use function $\varphi$ in \eqref{varphi} to write
\begin{equation*}
\begin{split}
A_{2,1}=\|\mathcal{D}_\xi^\theta(|\xi|^{1+a}\hat{\phi})\| \leq \|\Dt(|\xi|^{1+a}\varphi(\xi)\ha)\|+\|\Dt(|\xi|^{1+a}(1-\varphi(\xi))\ha)\|
=:A_{2,1}^1+A_{2,1}^2.
\end{split}
\end{equation*}
From \eqref{Leib} we deduce 
 \begin{equation}
\begin{split}\label{A21}
A_{2,1}^1 & \lesssim \||\xi|^{1+a} \varphi(\xi)\Dt \ha\|+\|\ha \Dt (|\xi|^{1+a}\varphi(\xi))\|\\
& \lesssim \||\xi|^{1+a}\varphi(\xi)\|_{L^\infty_\xi} \|\Dt \ha\|+\|\ha\|\|\Dt (|\xi|^{1+a}\varphi(\xi))\|_{L^\infty_\xi}\\
&\lesssim \||x|^\theta\phi\|+\|\phi\|,
\end{split}
\end{equation}
where we used Proposition \ref{Dstein} to obtain that $\|\Dt (|\xi|^{1+a}\varphi(\xi))\|_{L^\infty_\xi}$ is finite.
Also, observing that the function $\xi\mapsto \frac{|\xi|^{1+a}(1-\varphi(\xi))}{\langle \xi \rangle^{1+a}}$ satisfies the assumptions in Proposition \ref{Jota}, from \eqref{Jotaf1} we obtain
\begin{equation}
\begin{split}\label{A211}
A_{2,1}^2=\Big\|J_\xi^\theta \Big( \frac{|\xi|^{1+a}(1-\varphi(\xi))}{\langle \xi \rangle^{1+a}}\langle \xi \rangle^{1+a}\ha\Big)\Big\|
\lesssim \|J_\xi^\theta (\lan^{1+a}\ha)\|.
\end{split}
\end{equation}
An application of Lemma \ref{interx} gives
$$
A_{2,1}^2\lesssim \|J_x^{(1+a)(1+\theta)}\phi\|+\|\lanx^{1+\theta}\phi\|
$$
and we deduce that
\begin{equation}\label{A21estimate}
A_{2,1}\lesssim  \|J_x^{(1+a)(1+\theta)}\phi\|+\|\lanx^{1+\theta}\phi\|.
\end{equation}
Let us now estimate $A_3$. By recalling that $\partial_{\xi}\hat{\phi}=-\widehat{ix\phi}$  we use Lemma \ref{DF} to  write
\begin{equation*}
\begin{split}
A_3 =\|\mathcal{D}_\xi^\theta(\psi\, \widehat{ix\phi})\|& \lesssim  \rho(t) \Big(\|x\phi\|+\|D^{2\theta}_y(x\phi)\|+\underbrace{\|D_x^{(1+a)\theta} (x\phi)\|}_{B}\Big)+\||x|^\theta x \phi\|\\
&\lesssim  \rho(t)\Big(\|J_y^{2\theta}(\lanx\phi)\|+B \Big)+\||x|^{\theta+1} \phi\|\\
&\lesssim  \rho(t)\Big(\|J_y^{2(1+\theta)}\phi\|+\|\lanx^{1+\theta}\phi\|+B \Big)+\||x|^{\theta+1}  \phi\|,
\end{split}
\end{equation*}
where we also used Lemma \ref{interx} in the last inequality.
Using Lemma \ref{interx} again, the term $B$ can be estimated as follows:
\begin{equation*}
\begin{split}
B &\leq \|J_x^{(1+a)\theta}(x\phi)\|
\leq \|\lan^{(1+a)\theta}\p_\xi \ha\|\\
&\lesssim \|\lan^{(1+a)\theta-1}\ha\|+\|J_\xi(\lan^{(1+a)\theta}\ha)\|\\
&\lesssim \|J_x^{(1+a)\theta}\phi\|+\|J_\xi^{1+\theta}\ha\|+\|\lan^{(1+\theta)(1+a)}\ha\|\\
&\lesssim \|J_x^{(1+a)(1+\theta)}\phi\|+\|\lanx^{1+\theta}\phi\|.
\end{split}
\end{equation*} 
For $A_1$, using Lemma \ref{DF} and Young's inequality we have
\[
\begin{split}
A_1&=\|\mathcal{D}_\xi^\theta(\psi\, \widehat{\partial_y^2\phi})\|\\
&\lesssim \rho(t)\left(\|\partial_y^2\phi\|+\|D_y^{2\theta}\partial_y^2\phi\|+\|D_x^{(1+a)\theta}\partial_y^2\phi\|\right)+\||x|^\theta\partial_y^2\phi\| \\
& \lesssim \rho(t)\left(\|J_y^{2(1+\theta)}\phi\|+\|J_x^{(1+a)(1+\theta)}\phi\|\right)+\||x|^\theta\partial_y^2\phi\|.
\end{split}\]
The last term in the above inequality can be estimated using Lemma \ref{interx},
$$
\||x|^\theta\partial_y^2\phi\|\leq \|J^2_y(\langle x\rangle^\theta\phi)\|\lesssim \|J_y^{2(1+\theta)}\phi\|+\|\langle x \rangle^{1+\theta}\phi\|,
$$
from which we obtain
$$
A_1\lesssim \rho(t)\left(\|J_y^{2(1+\theta)}\phi\|+\|J_x^{(1+a)(1+\theta)}\phi\|\right)+\|\langle x \rangle^{1+\theta}\phi\|.
$$
Gathering together the above estimates for $A_1$, $A_2$, and $A_3$, we then infer
\begin{equation}\label{r1esti1}
\begin{split}
\||x|^{1+\theta}U(t)\phi\|\lesssim\rho_1(t)\big(\|\phi\|_{H^{2(1+\theta)}_y}+\|\phi\|_{H^{(1+a)(1+\theta)}_x}+\||x|^{1+\theta}\phi\|\big),
\end{split}
\end{equation}
where $\rho_1$ is a continuous increasing function on $t\in [0,T]$.

 Now using \eqref{r1esti1} in \eqref{r1int}, we obtain for all $t\in [0,T]$ 
 \begin{equation}\label{int3}
 \begin{split}
\||x|^{r_1} u(t)\|&\leq \||x|^{r_1} U(t)\phi\|+\int_0^t \||x|^{r_1} U(t-\tau)z(\tau)\|d\tau\\
&\lesssim\rho_1(T)(\|\phi\|_{H^{2r_1}_y}+\|\phi\|_{H^{(1+a)r_1}_x}+\||x|^{r_1}\phi\|)+\\
&\quad+\int_0^t \rho_1(t-\tau)\big(\|\p_x u^2(\tau)\|_{H^{2r_1}_y}+\|\p_x u^2(\tau)\|_{H^{(1+a)r_1}_x}+\||x|^{r_1}\p_x u^2(\tau)\|\big)d\tau.
\end{split}
\end{equation}
Note that
\begin{equation*}
\begin{split}
\|\p_x u^2\|_{H^{2r_1}_y}\lesssim \|J_x^{(1+a)r_1+1}u^2\|+\|J_y^{2r_1+\frac{2}{1+a}}u^2\|
\lesssim \|u\|^2_{H^{(1+a)r_1+1,2r_1+\frac{2}{1+a}}},
\end{split}
\end{equation*}
where we used that $H^{(1+a)r_1+1,2r_1+\frac{2}{1+a}}$ is a Banach algebra.
Our assumption $s\geq r_1+\frac{1}{1+a}$ implies  $E^s\hookrightarrow H^{(1+a)r_1+1,2r_1+\frac{2}{1+a}}$ and we deduce  
$$
\|\p_x u^2\|_{H^{2r_1}_y}\lesssim M^2.
$$
A similar argument also show that
$$
\|\p_x u^2\|_{H^{(1+a)r_1}_x}\lesssim M^2.
$$
and
$$
\||x|^{r_1}\p_x u^2\|\lesssim \|\p_xu\|_{L^\infty}\||x|^{r_1}u\|\lesssim M\||x|^{r_1}u\|.
$$
Consequently, from \eqref{int3} we deduce
\begin{equation}\label{int3.1}
\||x|^{r_1} u(t)\|\leq c+c\int_0^t(1+\||x|^{r_1}u(\tau)\|)d\tau, \qquad t\in[0,T].
\end{equation}
An application of Gronwall's lemma gives $\sup_{t\in[0,T]}\||x|^{r_1} u(t)\|<\infty$.

\noindent {\bf Case b).} $r_1=2$. In view of \eqref{F2},
\[
\begin{split}
\|x^2U(t)\phi\|&= \|\partial_{\xi}^2(\psi\hat{\phi})\|\\
&\lesssim t\Big( \||\xi|^a\hat{\phi}\|+\||\xi|^{2(1+a)}\hat{\phi}\|+\||\xi|^{1+a}\eta^2\hat{\phi}\| +\|\eta^4\hat{\phi}\|\\
&\quad \quad +\|\eta^2\partial_{\xi}\hat{\phi}\| +\||\xi|^{1+a}\partial_{\xi}\hat{\phi}\| \Big) +\|\partial_{\xi}^2\hat{\phi}\|.
\end{split}
\]
From Young's inequality and Lemma \ref{interx} it is not difficulty to obtain
\begin{equation*}
\begin{split}
\|x^2 U(t)\phi\|\lesssim\|\phi\|_{H_x^{2(1+a)}}+\|\phi\|_{H_y^{4}}+\|x^2 \phi\|, \quad t\in[0,T].
\end{split}
\end{equation*}
Using the same argument as in \eqref{int3} and \eqref{int3.1} we also deduce $\sup_{t\in[0,T]}\|x^{2} u(t)\|<\infty$. Part 2) is thus completed.\\
 
\noindent {\bf Part 3).} As we already said it suffices to show the persistence in $L^2_{r_1,0}$. So assume $r_1\in (2,5/2+a)$. Next we divide the proof into the cases  $0<a\leq 1/2$ and $1/2<a<1$.\\

\noindent {\bf Case a).} $0<a\leq 1/2$. Write  $r_1=2+\theta$, where $0<\theta<1/2+a$. In this case it is clear that $\theta\in(0,1)$.
Using \eqref{F2}, 
\begin{equation}
\begin{split}\label{D2}
\||x|^{2+\theta}U(t)\phi\|\lesssim_a& \ t\Big (\|\dte(\psi\sgn(\xi)|\xi|^a \ha)\|+\|\dte(\psi|\xi|^{2(1+a)} \ha)\|+\|\dte(\psi|\xi|^{1+a}\eta^2 \ha)\|+\\
&+\|\dte(\psi \eta^4 \ha)+\|\dte(\psi|\xi|^{1+a}\p_\xi \ha)\|+\|\dte(\psi \eta^2 \p_\xi \ha)\|\Big)+\|\dte(\psi \p_\xi^2 \ha)\|\\
&=:B_1+\ldots+B_7.
\end{split}
\end{equation}
Let us estimate each one of the terms $B_j$, $j=1,\ldots,7$. By Lemma \ref{DF} and Young's inequality
\begin{equation*}
\begin{split}
B_1 &\lesssim t\rho(t) \Big(\|D_x^{a} \h \phi\|+\|D_y^{2\theta}D_x^a \h \phi\|+\|D_x^{(1+a)\theta}D^a_x \h \phi\|\Big)+\||x|^{\theta}D^a \h \phi\|\\
&\lesssim t\rho(t)\Big(\|\phi\|_{H^{2(1+\theta)}_y}+\| \phi\|_{H^{(1+a)(1+\theta)}_x}\Big)+\underbrace{\||x|^{\theta}D^a_x \h \phi\|}_{K},
\end{split}
\end{equation*}
To estimate $K$, we make use of function $\chi$ in \eqref{chi} to  write
\begin{equation*}
\begin{split}
K=\|D_\xi^\theta (|\xi|^a \sgn(\xi)\ha)\|
\leq \|D^{\theta}_\xi (|\xi|^{a}\sgn(\xi)\chi \ha)\|+\|D^{\theta}_\xi (\underbrace{|\xi|^{a}\sgn(\xi)(1-\chi)\ha}_{L})\|=: K_{1}+K_{2}.
\end{split}
\end{equation*}
Thus, in view of \eqref{Leibh},
\begin{equation}\label{K1term}
\begin{split}
K_{1}& \lesssim \|\Dt (|\xi|^a \sgn(\xi)\chi \ha)\|\\
&\lesssim\||\xi|^a \sgn(\xi)\chi\|_\infty \|\Dt \ha\|+\|\ha\|_\infty \|\Dt (|\xi|^a \sgn(\xi)\chi)\|.
\end{split}
\end{equation}
Since $\chi(\xi,\eta)=\varphi(\xi)\varphi(\eta)$ the term $\||\xi|^a \sgn(\xi)\chi\|_\infty$ is clearly finite. Also, an application of Proposition \ref{Dstein} gives that $\|\Dt (|\xi|^a \sgn(\xi)\chi)\|$ is finite. It is to be clear that at this point the assumption $\theta<1/2+a$ is crucial. From Sobolev's embedding we then obtain
\begin{equation*}
K_1\lesssim \||x|^\theta \phi\|+\|\lanx^{1+\theta} \phi\|+\|\lany^{1+\theta} \phi\|.
\end{equation*}
For $K_2$ we use the inequality $\|D^\theta_\xi L\|\leq \|L\|^{1-\theta}\|\partial_{\xi} L\|^\theta\lesssim \|L\|+\|\partial_{\xi} L\|$. The term $\|L\|$ is clearly finite. In addition, since $1-\chi$ vanishes around the origin,
\begin{equation*}
\begin{split}
\|\partial_\xi L \| &= \Big \|a\frac{1-\chi}{|\xi|^{1-a}}\ha-|\xi|^{a}\sgn(\xi)\partial_\xi \chi \ha+|\xi|^{a}\sgn(\xi)(1-\chi)\partial_\xi \ha\Big\|\\
&\lesssim \Big\|\mathcal{X}_{\{|\xi|,|\eta|\geq 1\}}\frac{1-\chi}{|\xi|^{1-a}}\ha\Big\|+\||\xi|^a\ha\|+\||\xi|^a \p_\xi \ha\|\\
&\lesssim  \|\ha\|+\||\xi|^a\ha\|+\||\xi|^a \p_\xi \ha\|,
\end{split}
\end{equation*}
where $\mathcal{X}_\Omega$ stands for the characteristic function of the set $\Omega$.
The first two terms in the above inequality are clearly finite. The last one may be estimated as follows: 
\begin{equation*}
\begin{split}
\||\xi|^a \p_\xi \ha\|\lesssim \|\partial_{\xi}\langle\xi\rangle^a\ha\|+\|J_\xi(\langle \xi \rangle^{a}\ha)\|\leq \|\phi\|+\|\langle \xi \rangle^{2a}\ha\|+\|J_\xi^2 \ha\|,
\end{split}
\end{equation*}
where we used that $\partial_{\xi}\langle\xi\rangle^a$ is bounded and Lemma \ref{interx}. Hence, we obtain
$$
K_2\lesssim \|\phi\|_{H^{2a}_x}+\|\lanx^2\phi\|,
$$ 
and consequently,
\begin{equation}\label{B1est}
B_1\lesssim t\rho(t)\Big(\|\phi\|_{H^{2(1+\theta)}_y}+\| \phi\|_{H^{(1+a)(1+\theta)}_x}\Big) +\|\phi\|_{H^{2a}_x}+\|\lanx^2\phi\| +\|\lany^{1+\theta} \phi\|.
\end{equation}
For $B_2$ we use Lemma \ref{DF} to get
\begin{equation*}
\begin{split}
B_2 &\lesssim  t\rho(t) (\|D_x^{2(1+a)}\phi\|+\|D_y^{2\theta} D_x^{2(1+a)}\phi\|+\|D_x^{(1+a)\theta}D_x^{2(1+a)}\phi\|)+
\||x|^\theta D_x^{2(1+a)}\phi\|.
\end{split}
\end{equation*}
The first three terms on the right-hand side of the above inequality can be estimated by using the Young inequality. The last one may be estimated as the term $A_{2,1}$ in \eqref{aa1}. Thus, we obtain
\begin{equation}
B_2\lesssim t\rho(t)\Big(\|\phi\|_{H^{2(2+\theta)}_y}+\| \phi\|_{H^{(1+a)(2+\theta)}_x}\Big)+\|\langle x\rangle^{2+\theta}\phi\|.
\end{equation}

Terms $B_3$ and $B_4$ are estimated similarly. Indeed, Lemma \ref{DF}, Young's inequality, \eqref{A21estimate}, and Lemma \ref{interx} yield
\begin{equation*}
\begin{split}
B_3&\lesssim  t\rho(t) \Big(\|D^{1+a}_x \p_y^2 \phi\|+\|D_y^{2\theta} D^{1+a}_x\p_y^2 \phi\|+\|D^{(1+a)\theta}_xD^{1+a}_x \p_y^2 \phi\|\Big)+\||x|^\theta D^{1+a}_x \p_y^2\phi\|\\
&\lesssim t\rho(t)\Big(\|\phi\|_{H^{2(2+\theta)}_y}+\| \phi\|_{H^{(1+a)(2+\theta)}_x}\Big)+ \|J_x^{(1+a)(1+\theta)}\partial_y^2\phi\|+\|\lanx^{1+\theta}\partial_y^2\phi\|\\
&\lesssim t\rho(t)\Big(\|\phi\|_{H^{2(2+\theta)}_y}+\| \phi\|_{H^{(1+a)(2+\theta)}_x}\Big)+
\|J_x^{(1+a)(2+\theta)}\phi\|+\|J_y^{2(2+\theta)}\phi\|+\|\lanx^{2+\theta}\phi\|,
\end{split}
\end{equation*}
and
\begin{equation*}
\begin{split}
B_4 &\lesssim  t\rho(t)\Big(\|\p_y^4 \phi\|+\|D_y^{2\theta} \p_y^4  \phi\|+\|D^{(1+a)\theta}\p_y^4 \phi\|\Big)+\||x|^\theta \p_y^4 \phi\|\\
& \lesssim t\rho(t)\Big(\|\phi\|_{H^{2(2+\theta)}_y}+\| \phi\|_{H^{(1+a)(2+\theta)}_x}\Big)+\|J_y^{2(2+\theta)}\phi\|+\|\lanx^{2+\theta}\phi\|.
\end{split}
\end{equation*}

Next, from Lemma \ref{DF} we get
\begin{equation*}
\begin{split}
B_5 &\lesssim  t\rho(t) \Big(\|D_x^{1+a}(x\phi)\|+\|D_y^{2\theta} D_x^{1+a}(x\phi)\|+\|D_x^{(1+a)(1+\theta)}(x\phi)\|\Big)+\||x|^\theta D_x^{1+a}(x\phi)\|\\
&\lesssim t\rho(t)\Big(\|J_y^{2(1+\theta)}(x\phi)\|+\|J_x^{(1+a)(1+\theta)}(x\phi)\|\Big)+\||x|^\theta D_x^{1+a}(x\phi)\|\\
&\lesssim t\rho(t) \Big(B_{5,1}+B_{5,2}\Big)+B_{5,3}
\end{split}
\end{equation*}
where we used Young's inequality to obtain
\begin{equation*}
\|D_y^{2\theta} D_x^{1+a}(x\phi)\|\lesssim \|J_y^{2(1+\theta)}(x\phi)\|+\|J_x^{(1+\theta)(1+a)}(x\phi)\|.
\end{equation*}
But, from Lemma \ref{interx},
\begin{equation*}
\begin{split}
B_{5,1}\lesssim \|J_y^{2(\theta+1)}(\lanx \phi)\|
\lesssim \|J_y^{2(2+\theta)}\phi\|+\|\lanx^{2+\theta}\phi\|.
\end{split}
\end{equation*}
Also, Plancherel's identity and Lemma \ref{interx} give
\begin{equation*}
\begin{split}
B_{5,2}&=\|\lan^{(1+a)(1+\theta)}\p_\xi \ha\|\\
&\lesssim \|\lan^{(1+a)(1+\theta)-1}\ha\|+\|J_\xi (\lan^{(1+a)(1+\theta)})\ha\|\\
&\lesssim \|J_x^{(1+a)(1+\theta)}\phi\|+\|J_\xi^{2+\theta}\ha\|+\|\lan^{(1+a)(2+\theta)}\ha\|\\
&\lesssim \|\lanx^{2+\theta}\phi\|+\|J_x^{(1+a)(2+\theta)}\phi\|.
\end{split}
\end{equation*}
Note that $B_{5,3}$ is exactly term $A_{2,1}$ in \eqref{aa1} with $x\phi$ instead of $\phi$. Thus, from \eqref{A21estimate}, we have
\begin{equation*}
\begin{split}
B_{5,3}\lesssim \|J_x^{(1+a)(1+\theta)}(x\phi)\|+\|\lanx^{1+\theta}x\phi\|
\lesssim \|\lanx^{2+\theta}\phi\|+B_{5,2}\lesssim \|\lanx^{2+\theta}\phi\|+\|J_x^{(1+a)(2+\theta)}\phi\|,
\end{split}
\end{equation*}
and conclude that
\begin{equation}\label{B5estimate}
B_5\lesssim t\rho(t)\Big(\|\phi\|_{H^{2(2+\theta)}_y}+\| \phi\|_{H^{(1+a)(2+\theta)}_x}+\|\lanx^{2+\theta}\phi\|\Big)+\|J_x^{(1+a)(2+\theta)}\phi\|+\|\lanx^{2+\theta}\phi\|.
\end{equation}
For $B_6$, Lemma \ref{DF} implies
\begin{equation}\label{B6}
\begin{split}
B_6 \lesssim  t\rho(t) \Big(\|\p_y^2(x\phi)\|+\|D_y^{2\theta} \p_y^2(x\phi)\|+\|D^{(1+a)\theta}_x\p_y^2(x\phi)\|\Big)+\||x|^\theta \p_y^2(x\phi)\|.
\end{split}
\end{equation}
From Lemma \ref{interx}, the first two terms on the right-hand side of \eqref{B6} may be estimated as
$$
\|\p_y^2(x\phi)\|+\|D_y^{2\theta} \p_y^2(x\phi)\|\lesssim \|J_y^{2(2+\theta)}\phi\|+\|\lanx^{2+\theta}\phi\|.
$$
For the third one we use Young's inequality to obtain
$$
\|D^{(1+a)\theta}_x\p_y^2(x\phi)\|\lesssim \|J_y^{2(1+\theta)}(x\phi)\|+\|J_x^{(1+a)(1+\theta)}(x\phi)\|\lesssim \|\lanx^{2+\theta}\phi\|+\|J_x^{(1+a)(2+\theta)}\phi\|+\|J_y^{2(2+\theta)}\phi\|,
$$
where we used the estimates for $B_{5,1}$ and $B_{5,2}$ above. Finally, using similar arguments,
\begin{equation*}
\begin{split}
B_7 &\lesssim  t\rho(t) \Big(\|x^2\phi\|+\|D_y^{2\theta} (x^2\phi)\|+\|D^{(1+a)\theta}_x(x^2\phi)\|\Big)+\||x|^\theta x^2\phi\|\\
&\lesssim t\rho(t) \Big(\|\lanx^{2+\theta}\phi\|+\|J_y^{2\theta} (\lanx^2 \phi)\|+\|J_x^{(1+a)(1+\theta)}(x\phi)\|+\|\lanx^{1+\theta}x\phi\|\Big)+\||x|^\theta x^2\phi\|\\
&\lesssim t\rho(t) \Big(\|\lanx^{2+\theta}\phi\|+\|J_y^{2(2+\theta)}\phi\|+B_{5,2}\Big).
\end{split}
\end{equation*}

Gathering together all the above inequalities, we deduce
\begin{equation}
\begin{split}\label{psi2}
\||x|^{2+\theta}U(t)\phi\|\lesssim\rho_2(t)\big(\|\phi\|_{H_x^{(1+a)(2+\theta)}}+\|\phi\|_{H_y^{2(2+\theta)}}+\||x|^{2+\theta}\phi\|+\||y|^{1+\theta}\phi\|\big),
\end{split}
\end{equation}
where $\rho_2$ is a continuous increasing function on $t\in [0,T]$.

Recalling that $r_1=2+\theta$, as in \eqref{r1int}, we then get
\begin{equation}\label{int33}
 \begin{split}
\||x|^{r_1} u(t)\|&\leq \||x|^{r_1} U(t)\phi\|+\int_0^t \||x|^{r_1} U(t-\tau)z(\tau)\|d\tau\\
&\lesssim\rho_2(T)(\|\phi\|_{H^{2r_1}_y}+\|\phi\|_{H^{(1+a)r_1}_x}+\||x|^{r_1}\phi\|)+\||y|^{r_1-1}\phi\|\\
& +\int_0^t \rho_2(t-\tau)(\|\p_x u^2(\tau)\|_{H^{2r_1}_y}+\|\p_x u^2(\tau)\|_{H^{(1+a)r_1}_x}+\||x|^{r_1}\p_x u^2(\tau)\|+\||y|^{r_1-1}\p_x u^2(\tau)\|)d\tau.
\end{split}
\end{equation}
Note that 
$$
\||y|^{r_1-1}\p_x u^2(\tau)\|\lesssim \|\partial_xu\|_{L^\infty}\||y|^{r_1-1} u(\tau)\|\lesssim M\sup_{t\in[0,T]}\||y|^{r_1-1} u(t)\|.
$$
The right-hand side of the above inequality is finite thanks to Case a) in Part 1). Thus, we have
\begin{equation*}
\begin{split}
\||x|^{r_1} u(t)\|\lesssim c
 +\int_0^t \rho_2(t-\tau)(\|\p_x u^2(\tau)\|_{H^{2r_1}_y}+\|\p_x u^2(\tau)\|_{H^{(1+a)r_1}_x}+\||x|^{r_1}\p_x u^2(\tau)\|)d\tau.
\end{split}
\end{equation*}
This last inequality is similar to that in \eqref{int3}. Consequently one can proceed as in Part 2) to get the desired. \\

\noindent {\bf Case b):}   $1/2<a<1$.  If $2<r_1<3$, by writing $r_1=2+\theta$,  we can use the same ideas as in Case a) to obtain the persistence. Note that in this case we also have $\theta<1<1/2+a$ and so we can still apply Proposition \ref{Dstein} to deduce that the term $\|\Dt (|\xi|^a \sgn(\xi)\chi)\|$ appearing in $K_1$ (see \eqref{K1term}) is finite. 

If  $r_1=3$, from \eqref{F3},
\begin{equation*}
\begin{split}\label{r1D3}
\|x^{3}U(t)\phi\| =\|\partial_{\xi}^3(\psi\hat{\phi})\|\lesssim_t \sum_{j=1}^{14}\|G_j\|
\end{split}
\end{equation*}
where the implicit constant depends continuously on $t\in[0,T]$. After several applications of Young's inequality and Lemma \ref{interx} it is not difficult to see that
$$
\sum_{j=1,j\neq6}^{14}\|G_j\|\lesssim \|J_y^{6}\phi\|+\|J_x^{3(1+a)}\phi\|+\|\lanx^{3}\phi\|.
$$
Moreover, if $\chi=\varphi(\xi)\varphi(\eta)$ denotes the function in \eqref{chi},
\begin{equation*}
\begin{split}
\|G_6\|&\lesssim \||\xi|^{a-1}\chi\hat{\phi}\|+\||\xi|^{a-1}(1-\chi)\hat{\phi}\|
\lesssim  \||\xi|^{a-1}\chi\|\|\hat{\phi}\|_{L^\infty_{\eta \xi}}+\||\xi|^{a-1}(1-\chi)\|_{L^\infty_{\eta \xi}}\|\hat{\phi}\|.
\end{split}
\end{equation*}
Since $\chi\equiv1$  near the origin, $\||\xi|^{a-1}(1-\chi)\|_{L^\infty_{\eta \xi}}$ is finite. Also, since $1/2<a<1$ the function $|\xi|^{a-1}\varphi(\xi)$ belongs to $L^2(\R)$, from which we deduce that $ \||\xi|^{a-1}\chi\|$ is finite. Consequently, from Sobolev's embedding,
\begin{equation}\label{G6est}
\|G_6\|\lesssim \|\hat{\phi}\|_{H^{1+}_{\xi \eta}}+\|\hat{\phi}\|\lesssim \|\lanx^{3}\phi\|+\|\lany^{r_2}\phi\|.
\end{equation}
From these estimates we obtain
$$
\|x^{3}U(t)\phi\|\lesssim_T \|J_y^{6}\phi\|+\|J_x^{3(1+a)}\phi\|+\|\lanx^{3}\phi\|+\|\lany^{r_2}\|,
$$
and we can proceed as before.

It remains to consider the case $3<r_1<5/2+a$. First we write $r_1=3+\theta$ with $1/2+\theta<a$. By using \eqref{F3} now we may write
\begin{equation}\label{D3s}
\begin{split}
	\||x|^{3+\theta}U(t)\phi\|\lesssim_{t} & \ \|\dte(\psi\sgn(\xi)\eta^2|\xi|^a \ha)\|+\|\dte(\psi\eta^2|\xi|^{2(1+a)} \ha)\|+\|\dte(\psi|\xi|^{3(1+a)} \ha)\|+\\
	&+\|\dte(\psi |\xi|^{1+a}\eta^4 \ha)+\|\dte(\psi\eta^6 \ha)\|+\|\dte(\psi |\xi|^{a-1} \ha)\|\\
	&+\|\dte(\psi \si|\xi|^{1+2a} \ha)\|+\|\dte(\psi\sgn(\xi)|\xi|^a \p_\xi\ha)\|+\|\dte(\psi |\xi|^{2(1+a)} \p_\xi\ha)\|+\\
	&+\|\dte(\psi \eta^2 |\xi|^{1+a} \p_\xi\ha)\|+\|\dte(\psi \eta^4 \p_\xi\ha)\|+\\
	&+\|\dte(\psi \eta^2 \p_\xi^2\ha)\|+\|\dte(\psi |\xi|^{1+a}\p_\xi^2\ha)\|+\|\dte(\psi \p_\xi^3\ha)\|\\
	=:&\,\,C_1+\cdots+C_{14},
\end{split}
\end{equation}
where the implicit constant depends continuously on $t\in[0,T]$. Using Young's inequality and Lemmas \ref{DF} and \ref{interx}, it is not difficult to deduce that
\begin{equation}\label{allCs}
\begin{split}
C_j \lesssim_{a,\theta,T} \|J_y^{2(3+\theta)}\phi\|+\|J_x^{(1+a)(3+\theta)}\phi\|+\|\lanx^{3+\theta}\phi\|, \ j=1,...,14 \ \mbox{and} \ j\not=6,7.
\end{split}
\end{equation}
What is left is to estimate $C_6$ and $C_7$. Let us start with $C_6$. Lemma \ref{DF} implies that
\begin{equation}
\begin{split}\label{c6}
C_6\lesssim \rho(t)\big(\|D^{a-1}_x\phi\|+\|D_y^{2\theta}D^{a-1}_x\phi\|+\|D^{(1+a)\theta}_xD^{a-1}_x\phi\|\big)+\underbrace{\||x|^\theta D^{a-1}_x\phi\|}_{E}.
\end{split}
\end{equation}
The term $\|D^{a-1}_x\phi\|$ may be estimated as in \eqref{G6est}. Now, with $\chi$ as in \eqref{chi},
$$
\|D_y^{2\theta}D^{a-1}_x\phi\|\lesssim \|\chi|\eta|^{2\theta}|\xi|^{a-1}\hat{\phi}\|+ \|(1-\chi)|\eta|^{2\theta}|\xi|^{a-1}\hat{\phi}\|=:L_1+L_2.
$$
But 
\begin{equation*}
L_1\leq \|\chi|\eta|^{2\theta}|\xi|^{a-1}\|\|\hat{\phi}\|_{L^\infty_{\eta \xi}}=\|\varphi(\eta)|\eta|^{2\theta}\|_{L^2_\eta}\|\varphi(\xi)|\xi|^{a-1}\|_{L^2_{\xi}}\|\hat{\phi}\|_{L^\infty_{\eta \xi}}\lesssim \|\hat{\phi}\|_{L^\infty_{\eta \xi}},
\end{equation*}
where we used that $|\xi|^{a-1}\varphi(\xi)$ belongs to $L^2(\R)$. Also, from Young's inequality,
$$
L_2\leq \left\|\frac{1-\chi}{\xi} \right\|_{L^\infty_{\eta \xi}}\||\eta|^{2\theta}|\xi|^{a}\hat{\phi}\|\lesssim  \||\xi|^{\frac{a(3+\theta)}{3}}\hat{\phi}\|+\||\eta|^{2(3+\theta)}\hat{\phi}\|\lesssim \|J_x^{(1+a)(3+\theta)}\phi\|+\|J_y^{2(3+\theta)}\phi\|.
$$
Thus, from Sobolev's embedding,
\begin{equation*}
\|D_y^{2\theta}D^{a-1}_x\phi\|\lesssim\|J_x^{(1+a)(3+\theta)}\phi\|+\|J_y^{2(3+\theta)}\phi\|+\|\lanx^{r_1}\phi\|+\|\lany^{r_2}\phi\|.
\end{equation*}
Clearly we have $\|D^{(1+a)\theta}_xD^{a-1}_x\phi\|\leq \|J_x^{(1+a)(3+\theta)}\phi\|$. For $E$ in \eqref{c6}, we write
\begin{equation}
\begin{split}\label{c6M}
E=\|D^{\theta}_\xi (|\xi|^{a-1}\hat{\phi})\|
\leq \|D^{\theta}_\xi (|\xi|^{a-1}\chi\hat{\phi})\|+\|D^{\theta}_\xi (|\xi|^{a-1} (1-\chi)\hat{\phi})\|
=: E_1+E_2,
\end{split}
\end{equation}
and split
\begin{equation*}
\begin{split}
E_1 \leq \|D^{\theta}_\xi (|\xi|^{a-1}\chi (\hat{\phi}(\xi,\eta)-\ha(0,\eta))\|+\|D^{\theta}_\xi (|\xi|^{a-1}\chi\hat{\phi}(0,\eta))\|
=: E_{1,1}+E_{1,2}.
\end{split}
\end{equation*}
By using the inequality $\|D_\xi^\theta f\|\lesssim \|f\|+\|\partial_{\xi}f\|$ and the mean value theorem, we deduce
\begin{equation*}
\begin{split}
E_{1,1} & \lesssim  \Big\||\xi|^{a-1}\chi (\hat{\phi}(\xi,\eta)-\ha(0,\eta))\Big\|+\Big\|\p_\xi \big(|\xi|^{a-1}\chi (\hat{\phi}(\xi,\eta)-\ha(0,\eta)\big)\Big\|\\
&\lesssim\Big\||\xi|^a \chi \frac{\hat{\phi}(\xi,\eta)-\hat{\phi}(0,\eta)}{|\xi|}\Big\|+\Big\||\xi|^{a-1} \chi \frac{\hat{\phi}(\xi,\eta)-\hat{\phi}(0,\eta)}{\xi}\Big\|+\\
&\quad+ \Big\||\xi|^a \p_\xi\chi\frac{\hat{\phi}(\xi,\eta)-\hat{\phi}(0,\eta)}{\xi}\Big\|+\Big\||\xi|^{a-1}\chi \p_\xi \ha\Big\|\\
&\lesssim \||\xi|^{a} \chi\|\|\p_\xi \ha\|_{L^\infty_{\xi \eta}}+\||\xi|^{a-1}\chi\|\|\p_\xi \ha\|_{L^\infty_{\xi \eta}}+\||\xi|^a \p_\xi \chi\|\|\p_\xi \ha\|_{L^\infty_{\xi \eta}}+\||\xi|^{a-1} \chi\|\|\p_\xi \ha\|_{L^\infty_{\xi\eta}}\\
& \lesssim \|\lanx^{r_1}\phi\|+\|\lany^{r_2}\phi\|.
\end{split}
\end{equation*}
Also,
\begin{equation*}
\begin{split}\label{Pxi}
 E_{1,2} \leq  \|\ha\|_{L^\infty_{\xi \eta}}\|D_\xi^\theta (|\xi|^{a-1}\chi)\|
\les  (\|\lanx^{r_1}\phi\|+\|\lany^{r_2}\phi\|)\|D_\xi^\theta (|\xi|^{a-1}\chi)\|
\les\|\lanx^{r_1}\phi\|+\|\lany^{r_2}\phi\|,
\end{split}
\end{equation*}
where we used Sobolev's embedding and  Proposition \ref{DsteinL3}, with $\gamma=a-1/2$ and $\epsilon=a-1/2-\theta$, to see that $\|D_\xi^\theta (|\xi|^{a-1}\chi)\|$ is finite.

Moreover, by setting $h(\xi,\eta)=|\xi|^{a-1}(1-\chi(\xi,\eta))$ it follows that $h,\p_\xi h \in L^\infty_{\eta \xi}$. Thus, from \eqref{Leib} and \eqref{Lei},
\begin{equation*}
\begin{split}\label{M2}
E_2 \lesssim \|\Dt h\|_{L^\infty_{\xi\eta}}\|\ha\|+\|h\|_{L^\infty_{\xi\eta}}\|\Dt \ha\|
\lesssim \|\phi\|+\||x|^\theta \phi\|,
\end{split}
\end{equation*}
which then gives that
$$
E\lesssim \|\lanx^{r_1}\phi\|+\|\lany^{r_2}\phi\|.
$$
Collecting the above estimates we finally conclude 
\begin{equation}\label{C6}
C_6\lesssim \|J_x^{(1+a)(3+\theta)}\phi\|+\|J_y^{2(3+\theta)}\phi\|+\|\lanx^{r_1}\phi\|+\|\lany^{r_2}\phi\|.
\end{equation}
 Next we estimate $C_7$. First we write 
 \begin{equation*}
 \begin{split}\label{c7}
 C_7&=\|\Dt(\psi \si|\xi|^{1+2a} \ha)\|\\
 &\leq \|\Dt(\chi \psi \si|\xi|^{1+2a}\ha)\|+\|\Dt((1-\chi) \psi \si|\xi|^{1+2a}\ha)\|\\
 &=:C_{7,1}+C_{7,2}.
 \end{split}
 \end{equation*}
In view of Lemma  \ref{Dchip} we promptly obtain
 \begin{equation}\label{C71}
 \begin{split}
 C_{7,1}\lesssim \|\phi\|+\||x|^\theta \phi\|.
 \end{split}
 \end{equation}
In addition, using interpolation and the definition of the function $\psi$,
 \begin{equation*}
 \begin{split}
 C_{7,2}&\lesssim  \|(1-\chi)\psi \si |\xi|^{1+2a}\ha\|+\|\p_\xi ((1-\chi) \psi \si|\xi|^{1+2a}\ha)\|\\
 &\lesssim \|D_x^{1+2a}\phi\|+\|\p_\xi \chi  \psi |\xi|^{1+2a}\ha\|+\|t(\eta^2-(2+a)|\xi|^{1+a})(1-\chi)|\xi|^{1+2a}\ha\|\\
 &\quad+\|(1-\chi)\psi |\xi|^{2a} \ha\|+ \|(1-\chi)\psi |\xi|^{1+2a}\p_\xi \ha\|\\
 &\lesssim_t \|D_x^{1+2a}\phi\|+\|\phi\|+\|D_y^2 D_x^{1+2a}\phi\|+\|D_x^{2+3a}\phi\|+\|J_x^{2a}\phi\|+\|J_\xi(\lan^{1+2a}\ha)\|.
 \end{split}
 \end{equation*}
In view of Lemma \ref{interx} and Young's inequality, 
 \begin{equation*}
 \begin{split}
 \|J_\xi(\lan^{1+2a}\ha)\|\lesssim \|J_\xi^{3+\theta}\ha\|+\|\lan^{\frac{(3+\theta)(1+2a)}{2+\theta}}\ha\|\lesssim \|J_x^{(1+a)(3+\theta)}\phi\|+\|\lanx^{3+\theta}\phi\|,
 \end{split}
 \end{equation*}
 and
 \begin{equation*}
 \begin{split}
 \|D_y^2 D_x^{1+a}\phi\|\lesssim \|D_y^{2(3+\theta)}\phi\|+\|D_x^{(1+2a)\frac{(3+\theta)}{2+\theta}}\phi\|\lesssim \|J_y^{2(3+\theta)}\phi\|+\|J_x^{(1+a)(3+\theta)}\phi\|,
 \end{split}
 \end{equation*}
 which implies
 \begin{equation}\label{C72}
 C_{7,2}\lesssim \|J_y^{2(3+\theta)}\phi\|+\|J_x^{(1+a)(3+\theta)}\phi\|+\|\lanx^{3+\theta}\phi\|.
 \end{equation}
 From \eqref{C71} and \eqref{C72}, we infer
 \begin{equation}\label{C7}
 \begin{split}
 C_7\lesssim_{T} \|J_y^{2(3+\theta)}\phi\|+\|J_x^{(1+a)(3+\theta)}\phi\|+\|\lanx^{3+\theta}\phi\|.
 \end{split}
 \end{equation}

Finally, from \eqref{D3s}, \eqref{allCs}, \eqref{C6}, and \eqref{C7}, we have
$$
\||x|^{3+\theta}U(t)\phi\|\lesssim_{T} \|J_x^{(1+a)(3+\theta)}\phi\|+\|J_y^{2(3+\theta)}\phi\|+\||x|^{r_1}\phi\|+\||y|^{r_2}\phi\|.
$$
As in \eqref{psi2} this last inequality is enough to apply Gronwall's  inequality and obtain the desired.\\

\noindent {\bf Part 4):} $r_1\in [5/2+a,7/2+a)$, $r_2>3$. Let us prove the persistence in $L^2_{r_1,0}$. We will divide into the cases $a\in(1/2,1)$ and $a\in(0,1/2]$ again.\\

\noindent {\bf Case a)} $a\in (1/2,1)$. Let us first suppose $3<r_1<4$ and write $r_1=3+\theta$, where $\theta\in [a-1/2,1)$. By using \eqref{F3} we obtain inequality \eqref{D3s}. Except for $C_6$ all other terms are estimated as in Part 3). So, what is left is to estimate $C_6$. At this point the assumption $\hat{\phi}(0,\eta)=0$ plays a crucial role. Indeed, 
Lemma \ref{DF} implies that
\begin{equation}
\begin{split}\label{c61}
C_6\lesssim \rho(t)\big(\|D^{a-1}\phi\|+\|D_y^{2\theta}D^{a-1}\phi\|+\|D^{(1+a)\theta}D^{a-1}\phi\|\big)+\underbrace{\||x|^\theta D^{a-1}\phi\|}_{E}.
\end{split}
\end{equation}
Following the same strategy as in Case b) of Part 3) we only need to estimate the term $E$. We split
\begin{equation}
\begin{split}\label{c6M1}
E=\|D^{\theta}_\xi (|\xi|^{a-1}\hat{\phi})\|
                           \leq \|D^{\theta}_\xi (|\xi|^{a-1}\chi\hat{\phi})\|+\|D^{\theta}_\xi (|\xi|^{a-1} (1-\chi)\hat{\phi})\|
                           =: E_1+E_2.
\end{split}
\end{equation}
For $E_2$ we follow the ideas above to conclude that $E_2 \lesssim \|\phi\|+\||x|^\theta \phi\|$. So we only need to take care of $E_1$. Here we cannot use the same strategy as in Case b) of Part 3) because in that case we strongly used that $\theta<a-1/2$. The idea here is to use the assumption $\hat{\phi}(0,\eta)=0$ and Taylor's theorem with integral remainder to write
\begin{equation}\label{taylor}
\hat{\phi}(\xi,\eta)=\xi \partial_\xi \hat{\phi}(0,\eta)+\int_0^\xi (\xi-\zeta)\partial_\xi ^2 \hat{\phi}(\zeta,\eta)d\zeta.
\end{equation}
Thus
\begin{equation*}
\begin{split}
E_1&\leq \|D_\xi^\theta (|\xi|^{a}\sgn(\xi)\chi\partial_\xi \hat{\phi}(0,\eta))\|+\Big\|D_\xi^\theta\Big(\underbrace{|\xi|^{a-1}\chi\int_0^\xi (\xi-\zeta)\partial_\xi ^2 \hat{\phi}(\zeta,\eta)d\zeta}_{N}\Big)\Big\|\\
&\leq \|\p_\xi \ha\|_{L^\infty_{\xi \eta}}\|D_\xi^\theta (|\xi|^{a}\sgn(\xi)\chi)\|+\|D_\xi^\theta N\|.
\end{split}
\end{equation*}
Since $\theta<a+1/2$, Proposition \ref{Dstein} and Stein derivative give that $\|D^\theta_\xi (|\xi|^a \sgn(\xi) \chi)\|$ is finite. By using the interpolation estimate $\|D_\xi^\theta N\|\leq \|N\|^{1-\theta}\|\p_\xi N\|^\theta$, we estimate
\begin{equation*}
\begin{split}
\|N\|
                            &\leq \|\partial_\xi^2 \hat{\phi}\|_{L^{\infty}_{\xi\eta}} \left\| |\xi|^{a-1}\chi\int_0^\xi (\xi-\zeta)d\zeta\right \|\\
                            &\lesssim \|\partial_\xi^2 \hat{\phi}\|_{L^{\infty}_{\xi\eta}} \||\xi|^{a+1}\chi \|\\
                            &\lesssim  \|\partial_\xi^2 \hat{\phi}\|_{L^\infty_{\xi \eta}}\\
                            &\lesssim \|\lanx^{r_1}\phi\|+\|\langle y \rangle^{r_2}\phi\|,
\end{split}
\end{equation*}
where in the last inequality we used Sobolev's embedding and that $r_1,r_2>3$.
Also,
\begin{equation*}
\begin{split}
\|\partial_\xi N\|&\leq \Big\| \partial_\xi \chi|\xi|^{a-1}\int_0^\xi (\xi-\zeta)\partial_\xi^2 \hat{\phi}(\zeta,\eta)d\zeta \Big\|+
                     \Big\|\chi \sgn(\xi) |\xi|^{a-2}\int_0^\xi (\xi-\zeta)\partial_\xi^2 \hat{\phi}(\zeta,\eta)d\zeta \Big\|+\\
                    &\quad +\Big\|\chi |\xi|^{a-1}\int_0^\xi \partial_\xi^2 \hat{\phi}(\zeta,\eta)d\zeta \Big\|\\
                            &\leq \Big(\| \p_\xi\chi|\xi|^{a-1}\xi^2 \|+\|\chi \sgn(\xi)|\xi|^{a-2}\xi^2\|+\|\chi |\xi|^{a-1}|\xi|\|\Big)\|\partial_\xi^2 \hat{\phi}\|_{L^{\infty}_{\xi\eta}}\\
                            &\lesssim  \|\partial_\xi^2 \hat{\phi}\|_{L^{\infty}_{\xi\eta}}\\
                            &\lesssim \|\lanx^{r_1}\phi\|+\|\langle y \rangle^{r_2}\phi\|.
\end{split}
\end{equation*}
Consequently,
\begin{equation*}
\begin{split}
E\lesssim \|\lanx^{r_1}\phi\|+\|\langle y \rangle^{r_2}\phi\|.
\end{split}
\end{equation*}
 and 
\begin{equation*}
\begin{split}
C_6 \lesssim  \|J_x^{(1+a)(3+\theta)}\phi\|+\|J_y^{2(3+\theta)}\phi\|+\|\lanx^{r_1}\phi\|+\|\lany^{r_2}\phi\|.
\end{split}
\end{equation*}
Therefore, also here we obtain the estimate
\begin{equation*}
\begin{split}
\||x|^{3+\theta}U(t)\phi\|\lesssim_{T}& \|J_y^{2(3+\theta)}\phi\|+\|J_x^{(1+a)(3+\theta)}\phi\|+\|\lanx^{3+\theta}\phi\|+\|\lany^{r_2}\phi\|,
\end{split}
\end{equation*}
which is enough to conclude the desired.

Next we suppose $r_1=4$. In this case we get the inequality
\begin{equation}
\begin{split}\label{psir14}
\|x^4 U(t)\phi\|\lesssim_{a,\theta,T}& \|J_y^{8}\phi\|+\|J_x^{4(1+a)}\phi\|+\|\lanx^{4}\phi\|+\|\lany^{r_2}\phi\|. 
\end{split}
\end{equation}
Indeed, to obtain \eqref{psir14} we use identity \eqref{F5}. We will present the estimate  only for the terms $H_1$, $H_6$ and $H_{12}$ in \eqref{F5}. To deal with the terms $H_j$, $j\neq 1,6, 12$ it is enough to use Plancherel's identity, Young's inequality and Lemma \ref{interx}.  Using function $\chi$ we write
\begin{equation}
\begin{split}\label{H1}
\|H_1\|&\les \|\eta^2 |\xi|^{a-1}\ha\|\\
&\les \|\chi \eta^2 |\xi|^{a-1}\ha\|+\|(1-\chi)\eta^2 |\xi|^{a-1}\ha\|\\
&\les\|\chi\eta^2|\xi|^{a-1}\|\|\ha\|_{L^\infty_{\xi\eta}}+\Big \|\frac{1-\chi}{\xi}\Big\|_{L^\infty_{\xi \eta}}\|\eta^2 |\xi|^a \ha\|\\
&\les \La +\|\eta^4 \ha\|+\||\xi|^{2a}\ha\|\\
&\les \|J_y^{8}\phi\|+\|J_x^{4(1+a)}\phi\|+\|\lanx^{4}\phi\|+\|\lany^{r_2}\phi\|,
\end{split}
\end{equation}
where   we used the Sobolev's embedding and  assumption $a>1/2$ to conclude that $\chi\eta^2|\xi|^{a-1}\in L^2(\R^2)$.
Also using Taylor's formula \eqref{taylor}, we similarly obtain
\begin{equation}
\begin{split}\label{H6}
\|H_6\|&\les \||\xi|^{a-2}\ha\|\\
&\les\||\xi|^{a-2}\chi \ha\|+\||\xi|^{a-2}(1-\chi) \ha\|\\
&\les\|\chi |\xi|^{a-1}\p_\xi \ha(0,\eta)\|+\Big\||\xi|^{a-2}\chi \int_0^\xi (\xi-\zeta)\p_\xi^2\ha(\zeta,\eta)d\zeta\Big\|+\Big \|\frac{1-\chi}{\xi^2}\Big\|_{L^\infty_{\xi \eta}}\||\xi|^a \ha\|\\
&\les \|\chi |\xi|^{a-1}\| \|\p_\xi \ha\|_{L^\infty_{\xi \eta}}+\|\chi |\xi|^a\|\|\p_\xi^2\ha\|_{L^\infty_{\xi \eta}}+\|D_x^a \phi\|\\
&\les\La+\|J_x^a \phi\|.
\end{split}
\end{equation}
The term $H_{12}$ can be estimated as
\begin{equation}
\begin{split}\label{H14}
\|H_{12}\|&\les \||\xi|^{a-1}\chi \p_\xi \ha\|+\||\xi|^{a-1}(1-\chi) \p_\xi \ha\|\\
&\les\||\xi|^{a-1}\chi\|\|\p_\xi \ha\|_{L^\infty_{\xi \eta}}+\Big \|\frac{1-\chi}{\xi}\Big\|_{L^\infty_{\xi \eta}}\||\xi|^a \p_\xi \ha\|\\
&\les\|\lanx^{r_1}\phi\|+\|\lany^{r_2}\phi\|+\|J_x^{2a}\phi\|.
\end{split}
\end{equation}

Next we consider the case  $4<r_1 <7/2+a$. Here we write  $r_1=4+\theta$ with $\theta<a-1/2$. Thus, from \eqref{F5} and Lemma \ref{DF}, we can use the ideas employed above to estimate $\||x|^{4+\theta}U(t)\phi\|$. Since all estimates demand too many calculation involving  Plancherel's identity, Young's inequality and Lemma \ref{interx} we will estimate only the terms  $\|\dte (\psi |\xi|^{a-1} \p_\xi \ha)\|$ and $\|\dte (\psi \sgn(\xi)|\xi|^{a-2}\ha)\|$, which present estimates slightly different and whose counterparts in \eqref{F5} are given by  $H_{12}$ and $H_{6}$, respectively. 

Using Lemma \ref{DF} we have
\begin{equation}
\begin{split}\label{h15}
\|\dte (\psi |\xi|^{a-1} \p_\xi \ha)\|&\les  \rho(t)\Big(\|D_x^{a-1}(x\phi)\|+\|D_y^{2\theta}D_x^{a-1} (x\phi)\|+\|D_x^{(1+a)\theta}D_x^{a-1}(x\phi)\|\Big)\\
&\quad+\underbrace{\||x|^\theta D_x^{a-1}(x\phi)\|}_{D}.
\end{split}
\end{equation}
The only term that brings extra difficult in \eqref{h15} is $D$. Using function $\chi$, we split
\begin{equation*}
\begin{split}
D \les \|\dte (|\xi|^{a-1}\chi \p_\xi \ha)\|+\|\dte (|\xi|^{a-1}(1-\chi) \p_\xi \ha)\|=:D_{1,1}+D_{1,2}.
\end{split}
\end{equation*}
The estimate for $D_{1,2}$ is similar to that of $E_2$ in \eqref{c6M}. For $D_{1,1}$ we write
\begin{equation*}
\begin{split}
D_{1,1}\les \|\dte (\underbrace{|\xi|^{a-1}\chi (\p_\xi \ha(\xi,\eta)-\p_\xi \ha(0,\eta)}_{R})\|+\|\dte (|\xi|^{a-1}\chi \p_\xi \ha(0,\eta))\|
=:D_{1,1}^1+D_{1,1}^2.
\end{split}
\end{equation*}
For $D_{1,1}^1$ we will use the interpolation inequality $\|D_\xi^\theta R\|\leq \|R\|^{1-\theta}\|\p_\xi R\|^{\theta}$. But, from the mean value theorem and Sobolev's embedding we infer
\begin{equation}
\begin{split}\label{h151}
\|R\|\les \|\p_\xi^2 \ha\|_{L^\infty_{\xi \eta}}\||\xi|^a \chi\|\les \|\lanx^{r_1}\phi\|+\|\lany^{r_2}\phi\|
\end{split}
\end{equation}
and
\begin{equation}
\begin{split}\label{h152}
\|\p_\xi R\|&\les \Big\||\xi|^{a-1}\chi \frac{\p_\xi \ha(\xi,\eta)-\p_\xi\ha(0,\eta)}{|\xi|}\Big\|+\||\xi|^{a-1}\p_\xi \chi (\p_\xi \ha(\xi,\eta)-\p_\xi \ha(0,\eta))\|\\
&\quad+\||\xi|^{a-1}\chi \p_\xi^2 \ha\|\\
&\les \||\xi|^{a-1}\chi\|\big (\|\p_\xi \ha\|_{L^\infty_{\xi \eta}}+\|\p_\xi^2 \ha\|_{L^\infty_{\xi \eta}}\big)+\||\xi|^{a-1}\p_\xi \chi\|\|\p_\xi \ha\|_{L^\infty_{\xi \eta}}\\
&\les\|\lanx^{r_1}\phi\|+\|\lany^{r_2}\phi\|.
\end{split}
\end{equation}
From \eqref{h151} and \eqref{h152} we obtain $D_{1,1}^1\les \|\lanx^{r_1}\phi\|+\|\lany^{r_2}\phi\|$. Moreover, since by  Proposition \ref{DsteinL3} the quantity $\|\dte(|\xi|^{a-1}\chi)\|$ is finite, we have
\begin{equation*}
\begin{split}
D_{1,1}^2 \les \|\p_\xi \ha\|_{L^\infty_{\xi \eta}}\|\dte(|\xi|^{a-1}\chi)\|
\les  \|\p_\xi \ha\|_{L^\infty_{\xi \eta}}
\les  \|\lanx^{r_1}\phi\|+\|\lany^{r_2}\phi\|.
\end{split}
\end{equation*}

Next we estimate the term $\|\dte(\psi \sgn(\xi) |\xi|^{a-2}\ha)\|$. From Lemma \ref{DF},
\begin{equation*}
\begin{split}
\|\dte(\psi \sgn(\xi)|\xi|^{a-2}\ha)\|\les  \rho(t) (\|D_x^{a-2}\phi\|+\|D_y^{2\theta}D_x^{a-2}\phi\|+\|D_x^{(1+a)\theta}D_x^{a-2}\phi\|)+\||x|^\theta D_x^{a-2}\h\phi\|.
\end{split}
\end{equation*}
Let us estimate the last term by writing
\begin{equation*}
\begin{split}
\||x|^\theta D_x^{a-2}\h\phi\|\les \|\dte(\sgn(\xi)|\xi|^{a-2}\chi \ha)\|+\|\dte(\sgn(\xi)|\xi|^{a-2}(1-\chi)\ha)\|
=:E_1+E_2.
\end{split}
\end{equation*}
Estimate for $E_1$ may be performed by using Taylor's formula \eqref{taylor} and proceeding as above. For $E_2$, we use interpolation to obtain
\begin{equation*}
\begin{split}
E_2&\les \||\xi|^{a-2}(1-\chi)\ha\|+\|\p_\xi \big(\sgn(\xi)|\xi|^{a-2}(1-\chi)\ha\big)\|\\
&\les \Big\|\frac{1-\chi}{\xi^2} \Big\|_{L^\infty_{\xi \eta}}\||\xi|^a \ha\|+\Big\|\frac{1-\chi}{\xi^3}\Big\|_{L^\infty_{\xi \eta}} \||\xi|^a\ha\|+\\
&\quad+\||\xi|^{a-2}\p_\xi \chi\|_{L^\infty_{\xi \eta}}\| \ha\|+\Big\|\frac{1-\chi}{\xi^2}\Big\|_{L^\infty_{\xi \eta}} \||\xi|^a\p_\xi\ha\|\\
&\les \|\phi\|+\|D_x^a \phi\|+\|D_x^a(x\phi)\|\\
&\les \|J_x^2 \phi\|+\|\lanx^2 \phi\|.
\end{split}
\end{equation*}

After all estimates we arrive to the inequality
\begin{equation*}
\begin{split}\label{psi4}
\||x|^{4+\theta}U(t)\phi\|\lesssim_{a,\theta,T}& \|J_y^{2(4+\theta)}\phi\|+\|J_x^{(1+a)(4+\theta)}\phi\|+\|\lanx^{4+\theta}\phi\|+\|\lany^{r_2}\phi\|,
\end{split}
\end{equation*}
which is enough to our purpose again. This completes the proof in Case a).\\

\noindent {\bf Case b).} $0<a\leq 1/2$. Assume first  $a=1/2$. In this case we must have $r_1\in[3,4)$. The case $3<r_1<4$ was already treated in Case a). So we may assume $r_1=3$. Here the persistence follows from the inequality
\begin{equation}
\begin{split}\label{r1D3meio}
\|x^{3}U(t)\phi\|\lesssim_{a,\theta,T}& \|J_y^{6}\phi\|+\|J_x^{3(1+a)}\phi\|+\|\lanx^{3}\phi\|+\|\lany^{r_2}\phi\|.
\end{split}
\end{equation}
To obtain \eqref{r1D3meio} we  use identity \eqref{F3}. Estimates for the terms $G_j$, $j\neq 6$, follows as an application of Young's inequality, Plancherel's identity and Lemma \ref{interx}. The term $G_6$ is the only one that has a slightly different estimate. In fact, since that $\ha(0,\eta)=0$, for all $\eta\in \R$, from the mean value theorem we obtain   
\begin{equation}
\begin{split}\label{G6}
\|G_{6}\|\les \||\xi|^{-1/2}\hat \phi\|
\les \Big\|\frac{\ha(\xi,\eta) }{\xi}\Big\|_{L^\infty_{\xi\eta}}^{\frac{1}{2}}\|\hat\phi\|_{L^1_{\xi\eta}}^{\frac{1}{2}}
\les\|\p_\xi\ha\|_{L^\infty_{\xi\eta}}^{\frac{1}{2}}\|\hat\phi\|_{L^1_{\xi\eta}}^{\frac{1}{2}}.
\end{split}
\end{equation}
Since $\|\hat{\phi}\|_{L^1_{\xi\eta}}\lesssim \|\phi\|_{H^\sigma}$, where sigma $\sigma>1$ is arbitrary, we can use Sobolev's embedding and Young's inequality to obtain
$$
\|G_6\|\lesssim \|\lanx^{r_1}\phi\|+\|\lany^{r_2}\phi\|+\|\phi\|_{H^\sigma}.
$$
Thus, after all calculations we obtain a similar term as in \eqref{int33} with the additional term $\int_0^t\|\p_xu^2\|_{H^\sigma}d\tau$. However, by choosing $\sigma>1$ satisfying $1+\sigma\leq (1+a)s$ (this is always possible because $s>r_1>5/2$) we obtain
$$
\int_0^t\|\p_xu^2\|_{H^\sigma}d\tau\lesssim \int_0^T\|u\|_{H^\sigma}\|u\|_{H^{1+\sigma}}d\tau\lesssim M^2,
$$
where we used that $E^s\hookrightarrow H^{1+\sigma}$. Thus we still may apply Gronwall's lemma to conclude the result.

Assume now $0<a<1/2$. In this case, $r_1$ must range the interval $(5/2,4)$. The case $3\leq r_1<4$ has already been treated above. So, we may assume  $5/2+a<r_1<3$. The proof runs as in Part 2)(Case b). In fact, by setting $r_1=2+\theta$, with $1/2+a<\theta$, all terms in \eqref{D2} can be estimated as above, except the term $K_1=\|D^{\theta}_\xi (|\xi|^{a}\sgn(\xi)\chi \ha)\|$ in the decomposition of  $B_1$. Here it can be estimated as follows
\begin{equation*}
\begin{split}
K_1&=\|D^{\theta}_\xi (|\xi|^{a}\sgn(\xi)\chi\ha(\xi,\eta))\|\\
&\les \||\xi|^{a}\sgn(\xi)\chi\ha(\xi,\eta))\|+\|\p_\xi (|\xi|^{a}\sgn(\xi)\chi\ha(\xi,\eta)))\|\\
&\les \|\phi\|+\Big\||\xi|^a \chi \frac{\ha}{\xi}\Big\|+\||\xi|^a \p_\xi \chi \ha \|+\||\xi|^a \chi \p_\xi \ha\|\\
&\les \|\phi\|+\||\xi|^a \chi\|\|\p_\xi \ha\|_{L^\infty_{\xi \eta}}+\||\xi|^a \p_\xi \chi\|_{L^\infty_{\xi \eta}}\|\phi\|+\||\xi|^a \chi\|_{L^\infty_{\xi \eta}}\|\p_\xi \ha\|\\
&\les \|\lanx^{r_1}\phi\|+\|\lany^{r_2}\phi\|,
\end{split}
\end{equation*}
where we used Sobolev's embedding and the facts that $\p_\xi \sgn(\xi)=2\delta_\xi$ and $\ha\delta_\xi=0$. Here, $\delta_\xi \in \mathcal{S}'(\R^2)$ is the delta function in the $\xi$-direction defined by
\begin{equation*}
\langle \delta_\xi, f \rangle=\int f(0,\eta)d\eta, \ \mbox{for all} \ f\in \mathcal{S}(\R^2).
\end{equation*}

The proof of Theorem \ref{anisogbozk} is thus completed.

\section{Unique continuation principles} \label{uniquep}

This section is devoted to establish Theorems \ref{P1} and \ref{UCP}. As we already said, we
follow closely the arguments in \cite{FLP1}, where the authors proved a similar result for the dispersion generalized BO equation. The main idea is
to explore the behavior of the gBO-ZK in the $x$-direction, which, in some
sense, is similar to one presented by the dispersion generalized BO equation.

\begin{proof}[Proof of Theorem \ref{P1}]
Let us start by recalling that the solution of \eqref{gbozk} is given by
\begin{equation}\label{121}
u(t)=U(t)\phi -\int_{0}^{t}U(t-\tau)z(\tau)d\tau,
\end{equation}
where $z=\frac12 \p_x u^2$.

Assume first $0<a<1/2$. By introducing the parameter $\alpha$ such that $5/2+a=2+\alpha$ we have $\alpha\in(0,1)$ and $2+\alpha<3$.
Without loss of generality we assume $t_1=0.$

By multiplying \eqref{121} by $|x|^{2+\alpha}$ and using Fourier transform, we deduce
\begin{equation}\label{121.1}
\tilde{\chi} D_{\xi}^{\alpha}\partial_{\xi}^2(\widehat{u(t)})=\tilde{\chi} D_{\xi}^{\alpha}\partial_{\xi}^2(\psi(\xi,\eta, t)\hat{\phi}) +\int_0^t \tilde{\chi} D_{\xi}^{\alpha}\partial_{\xi}^2(\psi(\xi,\eta, t-\tau)\hat{z})\,d\tau,
\end{equation}
where  $\tilde{\chi}=\varphi(\xi)e^{-\eta^2}$ and $\psi$ is given in \eqref{psidef}. In view of \eqref{F2},  we may write the linear part in \eqref{121.1} as follows
\begin{equation*}
\begin{split}
\tilde{\chi} D_{\xi}^{\alpha}\partial_{\xi}^2(\psi(\xi,\eta, t)\hat{\phi})&= [\varphi;D_{\xi}^{\alpha}]\partial_{\xi}^2(e^{-\eta^2}\psi(\xi,\eta,t) \hat{\phi})+D_\xi^{\alpha}(\varphi \p_\xi^2 (e^{-\eta^2}\psi(\xi,\eta, t) \hat{\phi}))\\
&=: A(\xi,\eta,t,\hat{\phi})+\tilde B_1(\xi,\eta,t,\hat{\phi})+\cdots+\tilde B_7(\xi,\eta,t,\hat{\phi})
\end{split}
\end{equation*}
where $\tilde{B}_j:=D^\alpha_\xi(\tilde{\chi}(\xi,\eta) F_j(\xi,\eta,t,\hat{\phi}))$, with $F_j$ given in \eqref{F2}.

\begin{claim}\label{claim1}
For all $t\in [0,T]$, we have $A,\tilde B_j\in L^2$, where $j=2,...,7$. 
\end{claim}
Indeed, using Proposition \ref{C} with $\Phi=\varphi$ and identity \eqref{F2} we obtain 
\begin{equation}
\begin{split}\label{Aphi}
\|A\|&=\|\|[\varphi;D_{\xi}^{\alpha}]\partial_{\xi}^2(e^{-\eta^2}\psi(\xi,\eta,t) \hat{\phi})\|_{L^2_\xi}\|_{L^2_\eta}\\ 
&\les\||\xi|^a \ha\|+\||\xi|^{2(1+a)}\ha\|+\|\eta^2 e^{-\eta^2}|\xi|^{1+a}\ha\|+\|\eta^4 e^{-\eta^2}\ha\|+\|\p_\xi \ha\|+\\
&\quad +\|\eta^2 e^{-\eta^2}|\xi|^{1+a}\p_\xi \ha\|+\|\p_\xi^2 \ha\|\\
&\les \|\lan^{2(1+a)}\ha\|+\|J_\xi(\lan^{1+a}\ha)\|+\|\p_\xi^2 \ha\|\\
&\les \|J_x^{2(1+a)}\phi\|+\|x^2 \phi\|,
\end{split}
\end{equation}
where  we also used that $\|\eta^{2k} e^{-\eta^2}\|_{L^\infty_\eta}\les 1$, $k=1,2$. The right-hand side of \eqref{Aphi} is finite because $\phi\in \mathrm{Z}^s_{\frac{5}{2}+a,r_2}$. Here, and in the inequalities to follow, the implicit constant may depend on $t$.

With respect to $\tilde B_j$ we  only deal,  for instance, with $\tilde B_2$ and $\tilde B_7$. The other terms can be estimated in a similar way. From \eqref{equiv} and \eqref{xi}, we obtain
\begin{equation*}
\begin{split}
\|\tilde B_2\|=\|\dt (\tilde{\chi} \psi |\xi|^{2(1+a)}\ha)\|
\lesssim \|\phi\|+\| |x|^\alpha\phi\|.
\end{split}
\end{equation*}
and
\begin{equation*}
\begin{split}
\|\tilde{B}_7\|=\|\dt (\ti \psi \p_\xi^2 \ha )\|
\lesssim \|x^2\phi\|+\||x|^{2+\alpha}\phi\|
\les\|\lanx^{5/2+a}\phi\|.
\end{split}
\end{equation*}
This finishes the proof of Claim \ref{claim1}.\\


Using \eqref{F2} again, the integral part in \eqref{121.1} can be write as
\begin{equation*}
\begin{split}
\int_{0}^{t}\Bigg\{[\va;&D_{\xi}^{\alpha}]\p_\xi^2(e^{-\eta^2}\psi(\xi,\eta,t-\tau)\hat{z})+\sum_{j=1}^{7}D_{\xi}^{\alpha}\Big( \tilde{\chi}(\xi,\eta)F_j(\xi,\eta,t-\tau,\hat{z})  \Big)\Bigg\}d\tau\\
=:&\ \mathcal A+\mathcal B_1+\dots+\mathcal B_7.
\end{split}
\end{equation*}

\begin{claim}\label{claim2}
For any $t\in [0,T]$, we have $\mathcal A, \mathcal B_j \in L^2$, for $j=1,...,7$.
\end{claim}
In fact, we can proceed as in the proof of Claim \ref{claim1}. To estimate $\mathcal A$ it is enough to follow  \eqref{Aphi}, with $z$ instead of $\phi$ to obtain
\begin{equation}
\begin{split}\label{mathA}
\|[\va;D_{\xi}^{\alpha}]\p_\xi^2(e^{-\eta^2}\psi(\xi,\eta,t-\tau)\hat{z})\| \les & \|J_x^{2(1+a)}(uu_x)\|+\|\lanx^2 uu_x\|.
\end{split}
\end{equation} 
From fractional Kato-Ponce's inequality (see Remark 1.5 in \cite{dong}) and Sobolev's embedding
\begin{equation}
\begin{split}\label{1.5}
\|J_x^{2(1+a)}(uu_x)\|&\les \|uu_x\|+\|D_x^{2(1+a)}(uu_x)\|\\
&\les \|u\|_{L^\infty_{xy}}\|u_x\|+\|u_x\|_{L^\infty_{xy}}\|D_x^{1+2a}u_x\|+\|u_x\|_{L^\infty_{xy}}\|D_x^{2(1+a)}u\|\\
&\lesssim \|J_x^{2(1+a)} u\|^2\\
&\les  \|u\|_{E^s}^2.
\end{split}
\end{equation}
Also, from Holder's inequality and Sobolev's embedding,
\begin{equation}\label{ux}
\|\lanx^2 uu_x\|\les \|\lanx^2 u\|\|u_x\|_{L^\infty_{xy}}\les \|\lanx^2 u\|\|u\|_{E^s}.
\end{equation}
Thus, from \eqref{mathA}-\eqref{ux}, we obtain
\begin{equation}\label{uxx}
\begin{split}
\|\mathcal A\|\les \int_0^t \|u(\tau)\|_{\mathrm{Z}_{2,0}^s }^2 d\tau
\les_{T} \sup_{[0,T]}\|u\|_{\mathrm{Z}_{2,0}^s }^2.
\end{split}
\end{equation}
The right-hand side of \eqref{uxx} is finite taking into account that $u\in C([0,T]; \mathrm{Z}^s_{r_1,r_2})$.

Concerning the terms $\mathcal {B}_j$'s, we only deal with $\mathcal B_7$. The other terms can be estimated in an easier way. First note that from Sobolev's embedding and Lemma \ref{interx},
\begin{equation}
\begin{split}\label{xalpha}
\|\lanx^\alpha u\|_{L^\infty_{xy}}&\les \|J_x^{(1+a)s/2}(\lanx^\alpha u)\|+\|J_y^{s}(\lanx^\alpha u)\|\\
&\les \|J_x^{(1+a)s}u\|+\|J_y^{2s}u\|+\|\lanx^{2\alpha}u\|\\
&\les \|u\|_{\mathrm{Z}_{2,0}^s }.
\end{split}
\end{equation}
From \eqref{xi} and \eqref{xalpha} we get
\begin{equation*}
\begin{split}
\|\mathcal {B}_7\|&\les \int_0^t\Big(\|\Dta (\ti\psi(\xi,\eta, t-\tau)\p_\xi \widehat{u^2}))\|+\|\Dta (\ti\psi(\xi,\eta, t-\tau)\xi\p_\xi^2 \widehat{u^2})\|\Big)d\tau\\
&\les \int_0^t\big(\|x u^2\|+\||x|^\alpha xu^2\|+\|x^2 u^2\|+\||x|^\alpha x^2u^2\|\big)d \tau\\
&\lesssim\int_0^t \|\lanx^{2+\alpha}u^2\|d \tau\\
 &\les\int_0^t\|\lanx^2 u\|\|\lanx^\alpha u\|_{L^\infty_{xy}}d \tau\\
 &\les\int_0^t\|u\|_{\mathrm{Z}_{2,0}^s }^2d \tau.
\end{split}
\end{equation*}
As in \eqref{uxx} we obtain the desired. This finishes the proof of Claim \ref{claim2}.\\


Note that in Claim \ref{claim1} we do not estimate the term $\tilde{B}_1$. Actually, this term allow us to obtain the result. First note we can write
\begin{equation}\label{media0}
\begin{split}
\tilde{\chi}F_1&=\ tc_a|\xi|^a \sgn(\xi)\psi(\xi,\eta,t)\tilde{\chi}(\ha(\xi,\eta)-\ha(0,\eta))+tc_a|\xi|^a \sgn(\xi)\psi(\xi,\eta,t)\tilde{\chi}\ha(0,\eta)  \\
      & =:F_{1,1}+F_{1,2},
\end{split}
\end{equation}with $c_a=-i(1+a)(2+a)$. We claim that $\|D_\xi^{\alpha}(F_{1,1})\|$ is finite. Note that interpolation (in the $\xi$-variable) and Young's inequality give
\begin{equation}\label{kkkk}
\begin{split}
\|D_\xi^{\alpha}(F_{1,1})\|=\| \|D_\xi^{\alpha}(F_{1,1})\|_{L^2_{\xi}}\|_{L^2_{\eta}}
 \lesssim \| \|F_{1,1}\|_{L^2_{\xi}}^{1-\alpha} \|\partial_{\xi}F_{1,1}\|_{L^2_{\xi}}^{\alpha}\|_{L^2_{\eta}}
\lesssim \|F_{1,1}\|+\|\partial_{\xi}F_{1,1}\|.
\end{split}
\end{equation}
Thus, it suffices to show that the right-hand side of the last inequality is finite. It is easy to check that $F_{1,1}\in L^2$. In addition, by using the mean value theorem and Sobolev's embedding, we deduce
\begin{equation*}
\begin{split}
\|\p_\xi F_{1,1}\|&\les \Big \||\xi|^a \ti \frac{\hat \phi (\xi,\eta)-\hat \phi (0,\eta)}{\xi}\Big\|+\||\xi|^a \p_\xi \ti (\hat \phi (\xi,\eta)-\hat \phi (0,\eta))\|+\||\xi|^a \ti \p_\xi \hat \phi\|+\\
&\quad +\||\xi|^a \p_\xi \psi\ti (\hat \phi (\xi,\eta)-\hat \phi (0,\eta))\|\\
&\les  \||\xi|^a \ti\|\|\p_\xi \hat \phi\|_{L^\infty_{\xi \eta}}+\||\xi|^a \p_\xi \ti\|\|\hat \phi\|_{L^\infty_{\xi \eta}}+\||\xi|^a \ti\|\|\p_\xi \hat \phi\|_{L^\infty_{\xi \eta}}+\|\ti|\xi|^a \p_\xi \psi\|\|\ha\|_{\li}\\
&\les \|\p_\xi \hat \phi\|_{L^\infty_{\xi \eta}}+\|\hat \phi\|_{\li}\\
&\les \|\lanx^{r_1}\phi\|+\|\lany^{r_2}\phi\|.
\end{split}
\end{equation*}

Next, we write
\begin{equation}\label{desB5}
\begin{split}
F_{1,2}
&= tc_a (\psi(\xi,\eta,t) -1) |\xi|^{a}\sgn(\xi)\ha(0,\eta) \tilde{\chi}+tc_a|\xi|^a \sgn(\xi)\ha(0,\eta)\tilde{\chi}\\
&=: F_{1,2}^{1}+F_{1,2}^{2}.
\end{split}
\end{equation}
As above it is easy to check that  $\|F_{1,2}^{1}\|$ is finite. Therefore, putting $t=t_2$, from Claims \ref{claim1} and \ref{claim2} and our assumptions it must be the case that 
\begin{equation*}
 D^{\alpha}_{\xi}(F_{1,2}^{2})=D^{\alpha}_{\xi}\big( t_2 c_a \ha (0,\eta)e^{-\eta^2}|\xi|^a \sgn(\xi)\varphi(\xi)\big) \in L^2(\R^2).
\end{equation*}
Fubini's theorem and Theorem \ref{stein} imply that
\begin{equation}\label{teoP1e}
t_2 c_a e^{-\eta^2}\ha (0,\eta)\mathcal{D}^{\alpha}_{\xi}\big(|\xi|^a \sgn(\xi)\varphi (\xi)\big)\in L^{2}_{\xi}(\R),\quad \mathrm{a.e.} \ \eta \in \R.
\end{equation}
Taking into account that $\alpha=a+1/2$, an application of Proposition \ref{Dstein} yields
$$
\hat{\phi}(0,\eta)=0, \ \mathrm{a.e.} \quad  \eta \in \R.
$$
In view of \eqref{fourieru} the proof of the theorem is completed in this case.

The case $a\in [1/2,1)$ follows by writing $5/2+a=3+\alpha$, where $\alpha=a-1/2$ and applying similar ideas as above. In this case, instead of \eqref{F2} and Proposition \ref{Dstein}, identity \eqref{F3} and Proposition \ref{DsteinL2} must be used. This completes the proof of the theorem.
\end{proof}

\begin{proof}[Proof of Theorem \ref{UCP}]
First we deal with the case $a\in (1/2,1)$. Without loss of generality  we assume $t_1=0<t_2<t_3.$  By setting $\alpha=a-1/2$ it is seen that $4+\alpha=7/2+a$ with $\alpha\in (0,1/2)$. In addition, for any $r_1<7/2+a$ it follows that $u\in C \big([0,T];\mathrm{Z}^{s}_{r_1,r_2}\big)$.

Now multiplying \eqref{121} by $|x|^{7/2+a}$ and using Fourier transform  we may write
\begin{equation}
\begin{split}\label{intp}  
D_\xi^\alpha \p_\xi^4(\widehat{u(t)})&=D_\xi^\alpha \p_\xi^4(\psi(\xi,\eta,t)\ha)-\int_0^t D_\xi^\alpha \p_\xi^4(\psi(\xi,\eta,t-\tau)\hat{z}(\tau))d\tau,\\
\end{split}
\end{equation}
where, as before, $z=\frac12\partial_{x}u^{2}$. If $\chi$ is as in \eqref{chi}, then in view of \eqref{F5} we write the linear part of \eqref{intp} as
\begin{equation}
\begin{split}\label{chiD}
\chi D_{\xi}^{\alpha}\partial_{\xi}^4(\psi(\xi,\eta, t)\hat{\phi})&=\ [\va(\xi);D_{\xi}^{\alpha}]\partial_{\xi}^4(\psi(\xi,\eta,t) \va(\eta)\hat{\phi})+D_\xi^{\alpha}(\va(\xi) \p_\xi^4 (\psi(\xi,\eta, t) \va(\eta)\hat{\phi}))\\
&=:\ C+D_1+\cdots +D_{25},
\end{split}
\end{equation}
where $D_j:=D^\alpha_\xi(\chi(\xi,\eta) H_j(\xi,\eta,t,\hat{\phi}))$.

\begin{claim}\label{claim3}
For all $t\in [0,T]$ we have $C,D_j\in L^2$, where $j\in \{1,...,25\}$ and $j\neq 6,12$. 
\end{claim}
To prove the claim, in view of \eqref{F5} and Proposition \ref{C} we infer
\begin{equation}
\begin{split}\label{dsC}
\|C\|=&\,\|\|[\va(\xi);D_{\xi}^{\alpha}]\partial_{\xi}^4(\psi(\xi,\eta,t) \va(\eta) \hat{\phi})\|_{L^2_\xi}\|_{L^2_\eta}\\
\les&\, \|\eta^2\va(\eta)|\xi|^{a-1}\ha\|+\||\xi|^{2a}\ha\|+\|\eta^2\va(\eta)|\xi|^{1+2a}\ha\|+\||\xi|^{2+3a}\ha\|+\|\|\eta^4\va(\eta)|\xi|^{a}\ha\|+\\
&+\||\xi|^{a-2}\ha\|+\|\eta^4\va(\eta)|\xi|^{2(1+a)}\ha\|+\|\langle\eta\rangle^2\va(\eta)|\xi|^{3(1+a)}\ha\|+\||\xi|^{4(1+a)}\ha\|+\\
&+\|\eta^2\va(\eta)|\xi|^a \p_\xi \ha\|+\|\eta^6\va(\eta)|\xi|^{1+a}\ha\|+ \|\eta^8\va(\eta)\ha\|+\||\xi|^{a-1}\p_\xi \ha\|+\\
&+\|\eta^4\va(\eta)|\xi|^{1+a}\p_\xi \ha\|+\|\eta^2\va(\eta)|\xi|^{2(1+a)}\p_\xi \ha\|+\||\xi|^{3(1+a)}\p_\xi \ha\|+\|\eta^6\va(\eta)\p_\xi \ha\|+\\
&+\||\xi|^{2a+1}\p_\xi\ha\|+\||\xi|^a \p_\xi^2 \ha\|+\|\eta^4\va(\eta)\p_\xi^2 \ha\|+\||\xi|^{2(1+a)}\p_\xi^2 \ha\|+\|\eta^2\va(\eta)|\xi|^{1+a}\p_\xi^2 \ha\|+\\
&+\|\eta^2\va(\eta)\p_\xi^3 \ha\|+\||\xi|^{1+a}\p_\xi^3 \ha\|+\|\p_\xi^4\ha\|\\
\les&\, \||\xi|^{a-1}\ha\|+\||\xi|^{a-2}\ha\|+\|\lan^{4(1+a)}\ha\|+\||\xi|^{a-1}\p_\xi \ha\|+\|J_\xi(\lan^{3(1+a)}\ha)\|+\\
&+\|J_\xi^2(\lan^{2(1+a)}\ha)\|+\|J_\xi^3(\lan^{1+a}\ha)\|+\|\p_\xi^4\ha\|\\
\les&\, \underbrace{\||\xi|^{a-1}\ha\|}_{I}+\underbrace{\||\xi|^{a-2}\ha\|}_{J}+\underbrace{\||\xi|^{a-1}\p_\xi \ha\|}_{L}+\|\lan^{4(1+a)}\ha\|+\|J_\xi^4 \ha\|,
\end{split}
\end{equation}
where  we also used $\|\eta^{2k}\va(\eta)\|_{L^\infty}\les 1$, for $k=1,2,3,4$, and Lemma \ref{interx}.

To deal with terms $I$, $J$ and $L$ we may proceed as in \eqref{H1}, \eqref{H6} and \eqref{H14}, respectively, to obtain
\begin{equation}\label{IJL}
I,J,L\les\|\phi\|_{L^2_{r_1,r_2}}+\|J_x^{2a}\phi\|.
\end{equation} 
Thus, by \eqref{dsC} and \eqref{IJL}, 
\begin{equation}\label{IJL1}
\|C\|\les \|J_x^{4(1+a)}\phi\|+\|\lanx^4 \phi\|+\|\lany^{r_2}\phi\|.
\end{equation}
Since $\phi\in \mathrm{Z}_{7/2+a,r_2}^{s}$ we see that right-hand side of \eqref{IJL1} is finite. Next we deal with terms $D_j$. First, note that  Lemma \ref{Dchip} implies
\begin{equation}
\begin{split}\label{28}
\|D_{25}\|=\|\dt (\chi \psi \p_\xi^4 \ha )\|
\lesssim \|x^4\phi\|+\||x|^{4+\alpha}\phi\|
\les\|\lanx^{7/2+a}\phi\|.
\end{split}
\end{equation}
For the terms  $D_j$, $j\neq 1,5,14,19$, it is sufficient to follow an argument as in \eqref{28}. For  $D_5$, using \ref{Da} (with $\sigma_4=r_2$) we obtain
\begin{equation*}
\begin{split}
D_5 \lesssim\|\dt (\chi \psi \sgn(\xi) \eta^4 |\xi|^a  \ha)\|
\lesssim \|J_x^{2(1+a)}\phi\|+\|\lanx^2 \phi\|+\|\lany^{r_2}\phi\|.
\end{split}
\end{equation*}
To estimate $D_{16}$ and $D_{22}$ we use Lemma \ref{Dchip}, Plancherel's identity and Lemma \ref{interx}. In fact, by \eqref{Da} (with $\sigma_4=2$),
\begin{equation*}
\begin{split}
D_{14} &\les \|\dt (\chi \psi \eta^2 \sgn(\xi)|\xi|^a \p_\xi \ha)\|\\
&\lesssim \|J_x^{2(1+a)}(x\phi)\|+\|\lanx^2 x\phi\|+\|\lany^{2}x\phi\|\\
&\les\|J_\xi(\lan^{2(1+a)}\ha)\|+\|\lanx^3 \phi\|+\|\lany^3\phi\|\\
&\les\|J_\xi^2 \ha\|+\|\lan^{4(1+a)}\ha\|+\|\lanx^3 \phi\|+\|\lany^3\phi\|\\
&\les\|J_x^{4(1+a)}\phi\|+\|\lanx^3 \phi\|+\|\lany^3\phi\|.
\end{split}
\end{equation*}
Also, by \eqref{Da} (with $\sigma_4=3/2$),
\begin{equation*}
\begin{split}
D_{19} &\lesssim \|\dt (\chi\psi \sgn(\xi)|\xi|^a \p_\xi^2 \ha)\|\\
&\lesssim \|J_x^{2(1+a)}(x^2\phi)\|+\|\lanx^2 x^2\phi\|+\|\lany^{3/2}x^2\phi\|\\
&\les\|J_\xi^2(\lan^{2(1+a)}\ha)\|+\|\lanx^4 \phi\|+\|\lany^3\phi\|\\
&\les\|J_\xi^4 \ha\|+\|\lan^{4(1+a)}\ha\|+\|\lanx^4 \phi\|+\|\lany^3\phi\|\\
&\les\|J_x^{4(1+a)}\phi\|+\|\lanx^4 \phi\|+\|\lany^3\phi\|.
\end{split}
\end{equation*}
Finally, for $D_1$, our assumption and Theorem \ref{P1} imply that $\hat{\phi}(0,\eta)=0$. So, using \eqref{taylor} we obtain
\begin{equation}
\begin{split}\label{identQ}
D_1&=c_1\dt(\eta^2 \sgn(\xi)|\xi|^{a-1}\ha \chi \psi)\\
&= c_1\dt \Big(\eta^2 |\xi|^a \p_\xi \ha(0,\eta)\chi \psi\Big)+c_1\dt \Big(\underbrace{\eta^2 |\xi|^a \xi^{-1}\chi\psi\int_0^\xi (\xi-\zeta)\partial_\xi^2 \hat{\phi}(\zeta,\eta)d\zeta}_{Q}\Big)\\
&=:D_{1,1}+D_{1,2},
\end{split}
\end{equation}
where $c_1=4a(2+a)(1+a)t^2$. Now we write
\begin{equation*}
\begin{split}
D_{1,1}&=c_1 \dt(\underbrace{\eta^2 |\xi|^{a}\p_\xi\ha(0,\eta) \chi (\psi-1)}_{L})+c_1\dt(\eta^2 |\xi|^{a}\p_\xi\ha(0,\eta) \chi) \\
&=:D_{1,1}^1+D_{1,1}^2.
\end{split}
\end{equation*}
Recalling the standard inequality $|e^{ir}-1|\leq |r|$, for any $r\in\R$, we see that
\begin{equation}\label{phi}
|\psi-1|\leq t|\xi(\eta^2-|\xi|^{1+a})|.
\end{equation}
Thus using \eqref{phi} and Sobolev's embedding
\begin{equation*}
\begin{split}
\|L\|\lesssim \|\eta^2 |\xi|^a \p_\xi \ha(0,\eta)\chi(\psi-1)\|
\lesssim \|\eta^2 |\xi|^a \p_\xi \ha(0,\eta)\chi t \xi(\eta^2-|\xi|^{1+a})\|
\lesssim \|\p_\xi \ha\|_{L^\infty_{\xi\eta}}
\les\La,
\end{split}
\end{equation*}
and
\begin{equation*}
\begin{split}
\|\p_\xi L\|&\lesssim \|\eta^2 |\xi|^a \sgn(\xi) \p_\xi \ha(0,\eta)\chi\frac{\psi-1}{\xi}\|+\|\eta^2 |\xi|^a \p_\xi \ha \p_\xi \chi (\psi-1)\|\\
&\quad +\|\eta^2 |\xi|^a \p_\xi \ha(0,\eta)\chi t(\eta^2-(2+a)|\xi|^{1+a})\psi\|\\
&\lesssim \|\p_\xi \ha\|_{L^\infty_{\xi\eta}}\\
&\les \La.
\end{split}
\end{equation*}
Consequently, by using interpolation (see \eqref{kkkk}) we deduce that $D_{1,1}^1\in L^2$. On  the other hand, using \eqref{Dstein4},
\begin{equation*}
\begin{split}
\|D_{1,1}^2\|\lesssim \|\p_\xi \ha\|_{L^\infty_{\xi\eta}}\|\eta^2 \varphi(\eta)\|_{L^\infty_\eta}\|\dt (|\xi|^a \varphi(\xi))\|_{L^2_\xi}
\lesssim\|\p_\xi \ha\|_{L^\infty_{\xi\eta}}
\les \La.
\end{split}
\end{equation*}
This shows that $D_{1,1}\in L^2$. To see that $D_{1,2}$ also belongs to $L^2$, we note that
\begin{equation}
\begin{split}\label{N}
\|Q\|
                            &\leq \|\partial_\xi^2 \hat{\phi}\|_{L^{\infty}_{\xi\eta}}\left\| |\xi|^a \xi^{-1}\chi\psi\int_0^\xi (\xi-\zeta)d\zeta\right \|\\
                            &\leq \|\partial_\xi^2 \hat{\phi}\|_{L^{\infty}_{\xi\eta}} \||\xi|^{a+1}\chi \|\\
                            &\lesssim  \|\partial_\xi^2 \hat{\phi}\|_{L^\infty_{\xi \eta}}\\
                            &\lesssim \La,
\end{split}
\end{equation}
and
\begin{equation}
\begin{split}\label{pN}
\|\partial_\xi Q\|&\leq \Big\| \partial_\xi \chi \psi|\xi|^{a-1}\int_0^\xi (\xi-\zeta)\partial_\xi^2 \hat{\phi}(\zeta,\eta)d\zeta \Big\|+
                     \Big\|\chi \psi \sgn(\xi) |\xi|^{a-2}\int_0^\xi (\xi-\zeta)\partial_\xi^2 \hat{\phi}(\zeta,\eta)d\zeta \Big\|+ \\
                    &\quad+\Big\|\chi \psi |\xi|^{a-1}\int_0^\xi \partial_\xi^2 \hat{\phi}(\zeta,\eta)d\zeta \Big\|+\Big\|  \chi \partial_\xi \psi|\xi|^{a-1}\int_0^\xi (\xi-\zeta)\partial_\xi^2 \hat{\phi}(\zeta,\eta)d\zeta \Big\|\\
                            &\leq \Big(\| \p_\xi\chi|\xi|^{a-1}\xi^2 \|+\|\chi \sgn(\xi)|\xi|^{a-2}\xi^2\|+\|\chi |\xi|^{a-1}|\xi|\|+\|\chi \p_\xi \psi |\xi|^{a+1}\|\Big)\|\partial_\xi^2 \hat{\phi}\|_{L^{\infty}_{\xi\eta}}\\
                            &\lesssim  \|\partial_\xi^2 \hat{\phi}\|_{L^{\infty}_{\xi\eta}}\\
                            &\lesssim \La.
\end{split}
\end{equation}
Interpolation then gives $D_{1,2}\in L^2$. Therefore $D_1 \in L^2$ and the proof of Claim \ref{claim3} is completed.\\

Next we analyze the integral  part of \eqref{intp}. By using \eqref{F5} we see that it can written as  
\begin{equation}
\begin{split}\label{D4} 
\int_{0}^{t}\Big\{[\chi;&D_{\xi}^{\alpha}]\p_\xi^4(\psi(\xi,\eta,t-\tau)\hat{z})+D_{\xi}^{\alpha}\big(\chi \p_\xi^4(\psi(\xi,\eta,t-\tau)\hat{z})\big)\Big\}d\tau\\
&=\int_0^t[\chi;D_{\xi}^{\alpha}]\p_\xi^4(\psi(\xi,\eta,t-\tau)\hat{z})d\tau+\sum_{j=1}^{28}\int_0^t D_{\xi}^{\alpha}\big(\chi H_j(\xi,\eta,t-\tau,\hat{z})\big)d\tau\\
&=: \mathcal C+\mathcal D_1+\dots+\mathcal D_{25}.
\end{split}
\end{equation}

\begin{claim}\label{claim4}
For any $t\in [0,T]$, we have $\mathcal{C}, \mathcal{D}_j \in L^2$, for $j\in \{1,...,25\}$ and $j \neq 6,12$.
\end{claim}

The idea to prove the claim is similar to that in Claim \ref{claim2}. In fact, as in \eqref{dsC}, with $z=\frac12 \p_x u^2$ instead of $\phi$,
\begin{equation}
\begin{split}\label{math4}
\|[\va;D_{\xi}^{\alpha}]\p_\xi^4(\va(\eta)\psi(\xi,\eta,t-\tau)\hat{z})\|
\les \|J_x^{4(1+a)}(uu_x)\|+\|\lanx^4 uu_x\|+\|\lany^{r_2}uu_x\|.
\end{split}
\end{equation} 
By using  Remark 1.5 in \cite{dong} again, we deduce
\begin{equation}
\begin{split}\label{4a}
\|J_x^{4(1+a)}(uu_x)\|&\les \|uu_x\|+\|D_x^{4(1+a)}(uu_x)\|\\
&\les \|u\|_{L^\infty_{xy}}\|u_x\|+\|u_x\|_{L^\infty_{xy}}\|D_x^{3+4a}u_x\|+\|u_x\|_{L^\infty_{xy}}\|D_x^{4(1+a)}u\|\\
&\lesssim \|J_x^{4(1+a)} u\|^2\\
&\les  \|u\|_{\mathrm{Z}_{4,r_2}^s }^2.
\end{split}
\end{equation}
From Holder's inequality and Sobolev's embedding 
\begin{equation}\label{ux1}
\|\lanx^4 uu_x\|\les \|\lanx^4 u\|\|u_x\|_{L^\infty_{xy}}\les \|\lanx^4 u\|\|u\|_{E^s},
\end{equation}
and
\begin{equation}\label{uxr2}
\|\lany^{r_2}uu_x\|\les \|\lany^{r_2}u\|\|u_x\|_{L^\infty_{xy}}\les \|\lany^{r_2}u\|\|u\|_{E^s}.
\end{equation}
Then by \eqref{math4}--\eqref{uxr2}
\begin{equation*}
\begin{split}
\|\mathcal C\|\les\int_0^t \|u(\tau)\|_{\mathrm{Z}_{4,r_2}^s }^2 d\tau
 \les \sup_{[0,T]}\|u\|_{\mathrm{Z}_{4,r_2}^s }^2.
 \end{split}
\end{equation*}

With respect to $\mathcal D_j's$ we will only estimate  $\mathcal D_1$ and $\mathcal D_{25}$. The other terms can be treated as in Claim \ref{claim3}.
In view of \eqref{Da} (with $\sigma_4=2$),
\begin{equation*}
\begin{split}\label{D11}
\|\Dta(\chi \psi(\xi,\eta,t-\tau)\sgn(\xi)\eta^2 |\xi|^{a-1}\xi \widehat{u^2})\|&\les \|\Dta(\chi \psi(\xi,\eta,t-\tau)\eta^2 |\xi|^{a} \widehat{u^2})\|\\
&\les \|J_x^{2(1+a)}u^2\|+\|\lanx^2 u^2\|+\|\lany^2 u^2\|\\
&\les(\|J_x^{2(1+a)}u\|+\|\lanx^2 u\|+\|\lany^2 u\|)\|u\|_{L^\infty_{xy}}\\
&\les\|u\|_{\mathrm{Z}_{2,2}^s}^2,
\end{split}
\end{equation*}
where we also used the product estimate $\|J^\sigma_x(fg)\|\les \|f\|_\infty\|J^\sigma_xg\|+\|g\|_\infty\|J^\sigma_xf\|$, $\sigma>0$ (see, for instance Lemma X4 in \cite{KP} or Proposition 1.1 (page 105) in \cite{Taylor}).

Hence,
\begin{equation*}
\|\mathcal D_1\|\les \sup_{[0,T]}\|u\|_{\mathrm{Z}_{2,2}^s}^2.
\end{equation*}
Also, from \eqref{xi} and \eqref{xalpha},
\begin{equation*}
\begin{split}\label{29}
\|\Dta(\chi \psi \p_\xi^4 \hat z)\|&\les \|\Dta(\chi \psi \p_\xi^2 \widehat{ u^2})\|+\|\Dta(\chi \psi \p_\xi^3 \widehat{u^2})\|+\|\Dta(\chi \psi \xi \p_\xi^4 \widehat{u^2})\|\\
&\les \|x^2 u^2\|+\||x|^\alpha x^2 u^2\|+\|x^3 u^2\|+\||x|^\alpha x^3 u^2\|+\|x^4 u^2\|+\||x|^\alpha x^4 u^2\|\\
&\les \|\lanx^{4+\alpha}u^2\|\\
&\les\|\lanx^4 u\|\|\lanx^\alpha u\|_{L^\infty_{xy}}\\
&\les \|u\|_{\mathrm{Z}_{4,0}^s}^2,
\end{split}
\end{equation*}
implying that
\begin{equation*}
\|\mathcal D_{25}\|\les \sup_{[0,T]}\|u\|_{\mathrm{Z}_{4,0}^s}^2.
\end{equation*}
 This finishes the proof of Claim \ref{claim4}.\\

Next we will deal with terms $D_6$, $D_{12}$, $\mathcal{D}_6$ and $\mathcal{D}_{12}$. First, for  $c_6=-ia(a^2-1)(a+2)$, using \eqref{taylor} we write
\begin{equation}
\begin{split}\label{d6}
D_6&=c_6 t\dt(\sgn(\xi)|\xi|^{a-2}\ha \chi \psi)\\
&= c_6 t\dt \Big(|\xi|^{a-1} \chi\psi \p_\xi \ha(0,\eta)\Big)+c_6 t\dt \Big(\underbrace{\sgn(\xi)|\xi|^{a-2}\chi\psi\int_0^\xi (\xi-\zeta)\partial_\xi^2 \hat{\phi}(\zeta,\eta)d\zeta}_{R}\Big)\\
&=:D_{6,1}+D_{6,2},
\end{split}
\end{equation}
and decompose
\begin{equation}
\begin{split}\label{d62}
D_{6,1}&=tc_6 \dt(\underbrace{|\xi|^{a-1}\p_\xi\ha(0,\eta) \chi (\psi-1)}_{S})+tc_6 \dt(|\xi|^{a-1}\p_\xi\ha(0,\eta) \chi) \\
&=:D_{6,1}^1+D_{6,1}^2.
\end{split}
\end{equation}
Now, using \eqref{phi} and Sobolev's embedding we obtain
\begin{equation*}
\begin{split}\label{S}
\|S\|\lesssim t^2\||\xi|^{a} \chi(\eta^2-|\xi|^{1+a})\|_{L^\infty_{\xi\eta}}\|\p_\xi \ha\|
\lesssim \|x\phi\|,
\end{split}
\end{equation*}
and
\begin{equation*}
\begin{split}\label{dS}
\|\p_\xi S\|&\leq\||\xi|^{a-2}\p_\xi \ha(0,\eta)\chi(\psi-1)\|+\||\xi|^{a-1}\p_\xi\ha(0,\eta)\p_\xi\chi(\psi-1)\|+
\||\xi|^{a-1}\p_\xi\ha(0,\eta)\chi\p_\xi\psi\|\\
&\lesssim \Big(\||\xi|^{a-1}\chi(\eta^2-|\xi|^{1+a})\|+\||\xi|^{a}(\eta^2-|\xi|^{1+a})\p_\xi\chi \|\Big)\|\p_\xi \ha\|_{L^\infty_{\xi\eta}}\\
&\lesssim \|\p_\xi \ha\|_{L^\infty_{\xi\eta}}\\
&\les\La.
\end{split}
\end{equation*}
Hence interpolation gives that $D_{6,1}^1\in L^2$. By using similar arguments we obtain
\begin{equation}
\begin{split}\label{R}
\|R\|\leq \left\| |\xi|^{a-2}\psi \chi\int_0^\xi (\xi-\zeta)\partial_\xi^2 \hat{\phi}(\zeta,\eta)d\zeta\right\|
                            \leq  \|\partial_\xi^2 \hat{\phi}\|_{L^\infty_{\xi \eta}}\||\xi|^{a-2}\xi^2 \psi \chi\|
                            \lesssim \La,
\end{split}
\end{equation}
and
\begin{equation}
\begin{split}\label{pR}
\|\partial_\xi R\|&\leq \left\|\psi|\xi|^{a-3}\psi \chi\int_0^\xi (\xi-\zeta)\partial_\xi^2 \hat{\phi}(\zeta,\eta)d\zeta \right\|+
                     \left\||\xi|^{a-2}\psi\chi\int_0^\xi \partial_\xi^2 \hat{\phi}(\zeta,\eta)d\zeta\right\|+ \\
                    &\quad+\left\||\xi|^{a-2}\chi \p_\xi \psi\int_0^\xi (\xi-\zeta)\partial_\xi^2 \hat{\phi}(\zeta,\eta)d\zeta\right\|+\left\||\xi|^{a-2}\p_\xi\chi\psi\int_0^\xi (\xi-\zeta)\partial_\xi^2 \hat{\phi}(\zeta,\eta)d\zeta\right\|\\
                            &\leq \big(\||\xi|^{a-3}\xi^2 \chi \|+\||\xi|^{a}\p_\xi \psi \chi\|+\| |\xi|^{a}\p_\xi\chi\|\big)\|\partial_\xi^2 \hat{\phi}\|_{L^{\infty}_{\xi\eta}}\\
                            &\lesssim  \|\partial_\xi^2 \hat{\phi}\|_{L^{\infty}_{\xi\eta}}\\
                            &\lesssim \La,
\end{split}
\end{equation}
from which we also obtain $D_{6,2}\in L^2$.

For $D_{12}$, we first note that
\begin{equation}
\begin{split}\label{d14}
D_{12}&=tc_{12} \dt(|\xi|^{a-1}\p_\xi\ha \chi \psi)\\
&= tc_{12} \dt  \big(\underbrace{|\xi|^{a-1} \p_\xi \ha\chi(\psi-1)}_{W}\big)+tc_{12}\dt \big(|\xi|^{a-1}\p_\xi \ha \chi\big)\\
&=:D_{12}^1+D_{12}^2,
\end{split}
\end{equation}
where $c_{12}=-4ia(2+a)(1+a)$. But using \eqref{phi}
\begin{equation*}
\begin{split}
\|W\|\lesssim\||\xi|^{a-1} \p_\xi \ha\chi\xi(\eta^2-|\xi|^{1+a})\|
\lesssim\||\xi|^{a} \chi (\eta^2-|\xi|^{1+a})\|_{L^\infty_{\xi\eta}}\|\p_\xi \ha\|
\lesssim \|x \phi\|
\end{split}
\end{equation*}
and
\begin{equation*}
\begin{split}
\|\p_\xi W\|&\lesssim\||\xi|^{a-1}\p_\xi \ha\chi(\eta^2-|\xi|^{1+a})\|+\||\xi|^{a}\p_\xi^2 \ha\chi (\eta^2-|\xi|^{1+a})\|+\\
&\quad+\||\xi|^{a-1}\p_\xi \ha \chi\p_\xi\psi\|+\||\xi|^{a-1}\p_\xi\ha(\psi-1)\p_\xi \chi\|\\
&\lesssim \|\p_\xi \ha\|_{L^\infty_{\xi\eta}}+\|x^2\phi\|\\
&\les\La,
\end{split}
\end{equation*}
where we used that $a\in (1/2,1)$ to see that $|\xi|^{a-1}\in L^2_\xi$. Hence $D_{12}^1\in L^2$. 

We may also write
\begin{equation}
\begin{split}\label{d1322}
D_{12}^2&=tc_{12}\Big(\dt \big(\underbrace{|\xi|^{a-1}(\p_\xi \ha(\xi,\eta)-\p_\xi \ha(0,\eta)\chi}_{U} \big)+ \dt \big( |\xi|^{a-1}\p_\xi \ha(0,\eta)\chi\big)\Big)\\
&=:D_{12}^{2,1}+D_{12}^{2,2}.
\end{split}
\end{equation}
Then, following the arguments above,
\begin{equation}
\begin{split}\label{T}
\|U\|\leq \Big\| |\xi|^{a}\frac{\p_\xi \ha(\xi,\eta)-\p_\xi \ha(0,\eta)}{|\xi|}\chi\Big\|
                            \lesssim  \pphi \|\xi|^{a}\chi\|
                            \lesssim \pphi
                           \les\La,
\end{split}
\end{equation}
and
\begin{equation}
\begin{split}\label{pT}
\|\partial_\xi U\|&\leq \Big\||\xi|^{a-1} \frac{\p_\xi \ha(\xi,\eta)-\p_\xi \ha(0,\eta)}{|\xi|}\chi \Big\|+
                     \Big\||\xi|^{a-1}\partial_\xi^2 \hat{\phi}(\xi,\eta) \chi \Big\|+ \\
                    &\quad+\Big\||\xi|^{a-1}(\p_\xi \ha(\xi,\eta)-\p_\xi \ha(0,\eta)) \p_\xi\chi  \Big\|\\
                            &\leq \Big(\||\xi|^{a-1}\chi \|+\||\xi|^{a}\p_\xi\chi\|\Big)\|\partial_\xi^2 \hat{\phi}\|_{L^{\infty}_{\xi\eta}}\\
                            &\lesssim  \|\partial_\xi^2 \hat{\phi}\|_{L^{\infty}_{\xi\eta}},\\
                            &\lesssim \La.
\end{split}
\end{equation}
 Thus from \eqref{d1322}--\eqref{pT} and interpolation, it follows that $D_{12}^{2,1} \in L^2$.

From \eqref{D4} and proceeding similarly as in \eqref{d6} and \eqref{d14} we can write
\begin{equation*}
\mathcal D_{6}=\mathcal D_{6,1}^1+\mathcal D_{6,1}^2+\mathcal D_{6,2} \quad \mbox{and} \quad \mathcal D_{12}=\mathcal D_{12}^1+\mathcal D_{12}^{2,1}+\mathcal D_{12}^{2,2}.
\end{equation*}
Also, by using the above arguments, with $z$ instead of $\phi$ it is not difficult to  conclude that
 $$\mathcal D_{6,1}^1, \mathcal D_{6,2}, \mathcal D_{12}^1, \mathcal D_{12}^{2,1}\in L^2 .$$
Hence, putting $t=t_2$ and setting $\tilde{D}=D_{6,1}^2-\mathcal{D}_{6,1}^2+D_{12}^{2,2}-\mathcal{D}_{12}^{2,2}$, from \eqref{intp}, \eqref{chiD}, \eqref{D4}, Claims \ref{claim3} and \ref{claim4}, and gathering the information above, we obtain that
\begin{equation*}
D_{\xi}^{\alpha}\p_\xi^4 \hat{u}(\cdot,\cdot,t_2)\in L^2 (\R^2)
\end{equation*}
if and only if
\begin{equation}
\begin{split}\label{tD}
\tilde{D}&=c_6\Big(t_2\dt\Big(|\xi|^{a-1}\chi\p_\xi \ha(0,\eta)-|\xi|^{a-1}\chi\int_0^{t_2} (t_2-\tau)\p_\xi \hat{z}(0,\eta,\tau)d\tau \Big)\Big)\\
&\quad+c_{12}\Big(t_2\dt\Big(|\xi|^{a-1}\chi\p_\xi \ha(0,\eta)-|\xi|^{a-1}\chi\int_0^{t_2} (t_2-\tau)\p_\xi \hat{z}(0,\eta,\tau)d\tau\Big)
\Big)\\
&=(c_6+c_{12})\dt \Big(|\xi|^{a-1}\chi\Big(t_2\p_\xi \ha(0,\eta)-\int_0^{t_2} (t_2-\tau)\p_\xi \hat{z}(0,\eta,\tau\Big)d\tau\Big)\in L^2(\R^2).
\end{split}
\end{equation}

Now by using the definition of the Fourier transform and integration by parts we deduce
\begin{equation}\label{pxiz}
\p_\xi \hat z (0,\eta,\tau)=\frac{i}{2}\int e^{-i\eta y}u^2(x,y,\tau)dxdy.
\end{equation}
Also, from \eqref{gbozk}, it is easily seen that
\begin{equation}\label{dtgb}
\frac{d}{d\tau}\int xe^{-i\eta y}u(x,y,\tau)dxdy=\frac12\int e^{-i\eta y}u^2(x,y,\tau)dxdy, \quad \eta\in\R.
\end{equation}
By combining  \eqref{pxiz} and \eqref{dtgb}

\begin{equation}\label{dtxi}
\p_\xi \hat z (0,\eta,\tau)=i\frac{d}{d\tau}\int xe^{-i\eta y}u(x,y,\tau)dxdy.
\end{equation}
By the definition of the Fourier transform
\begin{equation}\label{pxip}
\p_\xi \ha(0,\eta)=-i\int xe^{-i\eta y}\phi(x,y)dxdy, \quad \mbox{for all} \quad \eta\in \R.
\end{equation}
Then, using \eqref{dtxi}, \eqref{pxip} and integrating by parts
\begin{equation*}\label{t2pxi}
\begin{split}
t_2\p_\xi \ha(0,\eta)-\int_0^{t_2} (t_2-\tau)\p_\xi \hat{z}(0,\eta,\tau)d\tau&=t_2\p_\xi \ha(0,\eta)-i\int_0^{t_2}(t_2-\tau)\frac{d}{d\tau}\int xe^{-i\eta y}u(x,y,\tau)dxdyd\tau\\
&=-i\int_0^{t_2}\int x e^{-i\eta y}u(x,y,\tau)dxdyd\tau.
\end{split}
\end{equation*}
By replacing the last identity in \eqref{tD} we obtain
\begin{equation*}
\dt (|\xi|^{a-1}\chi) \int_{0}^{t_2}\int x e^{-i\eta y}u(x,y,\tau)dxdyd\tau \in L^2(\R^2).
\end{equation*}
Therefore from Fubini's theorem and \eqref{equiv} (recall that $a-1=\alpha-\frac{1}{2}$)
\begin{equation*}
\mathcal{D}_\xi^{\alpha}(|\xi|^{\alpha-1/2}\varphi) \int_{0}^{t_2}\int x e^{-i\eta y}u(x,y,\tau)dxdyd\tau \in L^2_\xi, \ \ \mbox{a.e.} \ \ \eta\in \R.
\end{equation*}

Thus from Proposition \ref{DsteinL2} we obtain
\begin{equation}\label{tauu}
\int_{0}^{t_2}\int x e^{-i\eta y}u(x,y,\tau)dxdyd\tau=0, \ \ \mbox{a.e.} \ \ \eta\in \R.
\end{equation}
This last identity allows us to obtain $\tau_1\in(0,t_2)$ such that
\begin{equation}\label{tau1}
\int x u(x,y,\tau_1)dxdy=0
\end{equation}
Performing a similar analysis we may also find $\tau_2\in(t_2,t_3)$ such that
\begin{equation}\label{tau2}
\int x u(x,y,\tau_2)dxdy=0
\end{equation}
Using \eqref{dtgb} (with $\eta=0$), \eqref{tau1}, \eqref{tau2} and the fact that the $L^2$ norm is a conserved quantity for \eqref{gbozk} we conclude that $\|\phi\|=0$, implying the desired. This finishes the proof of the theorem \ref{UCP} in the case $a\in(1/2,1)$.

If $a=1/2$ then $7/2+a=4$. Hence, using \eqref{F5} and following the same strategy as above we arrive to
\begin{equation*}
|\xi|^{-1/2}\varphi(\xi) \int_{0}^{t_2}\int x e^{-i\eta y}u(x,y,\tau)dxdyd\tau \in L^2_\xi, \ \ \mbox{a.e.} \ \ \eta\in \R.
\end{equation*}
Since $|\cdot|^{-1/2}\varphi(\cdot)\notin L^2$, we also obtain \eqref{tauu}.

Finally, if $a\in (0,1/2)$ we write $7/2+a=3+\alpha$ and use \eqref{F3} to obtain an expression similar to \eqref{intp}. After some calculations and the help of Proposition \ref{Dstein} we may also obtain \eqref{tauu}. Since it demands too many calculations following the arguments above we will omit the details. The proof of Theorem \ref{UCP} is thus completed. 
\end{proof}

\section*{Acknowledgment}

A.P. is partially supported by CNPq/Brazil grant  	303762/2019-5.


\begin{thebibliography}{99}
\bibitem{BUM} E. Bustamante, J.J. Urrea, and J. Mej\'ia,
The Zakharov-Kuznetsov equation in weighted Sobolev spaces, \textit{J. Math. Anal. Appl.} 433, 149--175, 2016.
	
\bibitem{AP} A. Cunha and A. Pastor, The IVP for the Benjamin-Ono-Zakharov-Kuznetsov
equation in weighted Sobolev spaces, {\em J. Math. Anal. Appl.} 417, 660--693, 2014.

\bibitem{APlow} A. Cunha and A. Pastor, The IVP for the Benjamin-Ono-Zakharov-Kuznetsov
equation in low regularity Sobolev spaces, \textit{J. Differential Equations} 261, 2041--2067, 2016.

\bibitem{dBO} A. Cunha, The Cauchy Problem for Dissipative Benjamin-Ono equation in Weighted Sobolev spaces, arXiv:1912.12943v3.


\bibitem{Cald}
A. P. Calder\'on, Commutators of singular integral operators, {\em Proc. Natl. Acad. Sci. USA.} 53, 1092--1099, 1965.



\bibitem{Dawson} L. Dawson, H. McGahagan, G. Ponce, On the decay properties
of solutions to a class of {S}chr\"odinger equations, {\em Amer. Math. Soc.}
136, 2081--2090, 2008.




\bibitem{GermanPonce}
G.~Fonseca and G. Ponce,
\newblock The IVP for the Benjamin-Ono equation in weighted Sobolev spaces,
\newblock {\em J. Func. Anal.} 260, 436--459, 2011.


\bibitem{FLP1}
G. Fonseca, F. Linares, and G. Ponce,
\newblock The IVP for the dispersion generalized
Benjamin-Ono equation in weighted Sobolev
spaces,
\newblock {\em Ann. Inst. H. Poincar\'e Anal. Non Lin\'eaire} 30, 763--790, 2013.


\bibitem{FP}
G. Fonseca and M. Pachon,
\newblock Well posedness for the two dimensional generalized Zakharov-Kuznetsov equation anisotropic weighted Sobolev spaces,
\newblock \textit{J. Math. Anal. Appl.} 443, 566--584, 2016.

\bibitem{GHerr} A. Gr\"unrock and S. Herr, The Fourier restriction method norm for the Zakharov–Kuznetsov equation, \textit{Discrete Contin. Dyn. Syst.} 34, 2061--2068, 2014.

\bibitem{HLRKW} J. Hickman, F. Linares, O.G. Ria\~no, K. Rogers and J. Wright, On a higher dimensional version of the Benjamin-Ono equation, \textit{SIAM J. Math. Anal.} 51, 4544--4569, 2019.


\bibitem{Hille} E. Hille, Methods in Classical and Functional Analysis, Addison-Wesley Publishing
Co., 1972.



\bibitem{IKT} A. D. Ionescu, C.E. Kenig and D. Tataru, Global well-posedness of the KP-I initial-value
problem in the energy space, \textit{Invent. Math.} 173, 265--304, 2008.

\bibitem{Iorio}
R. J. Iorio,
\newblock On the Cauchy problem for the Benjamin-Ono equation,
\newblock {\em Comm. Partial Differential Equations} 11, 1031--1084, 1986.

\bibitem{Iorio1}
R. J. Iorio,
\newblock The Benjamin-Ono equation in weighted Sobolev spaces,
\newblock {\em J. Math. Anal. Appl.} 157, 577--590, 1990.

\bibitem{Iorio2} R. J. Iorio, Unique continuation principle for the Benjamin-Ono equation, \textit{Differential Integral Equations} 16, 1281--1291, 2003.

\bibitem{Jorge}
M. C. Jorge, G. Cruz-Pacheco, L. Mier-y-Teran-Romero and N. F. Smyth,
\newblock Evolution of twodimensional
lump nanosolitons for the Zakharov-Kuznetsov and electromigration equations,
\newblock {\em Chaos} 15, 2005, 037104.

\bibitem{dong}
D. Li, On Kato-Ponce and fractional Leibniz, 
\newblock {\em Rev. Mat. Iberoamericana} 35, 23--100, 2019.


\bibitem{pastran}
G. Fonseca, R. Pastr\'an and Guillermo Rodr\'iguez-Blanco
\newblock The IVP for a nonlocal perturbation of the Benjamin-Ono equation in classical and weighted Sobolev spaces,
\newblock {\em J. Math. Anal. and Appl.} 476, 391--425, 2019.

\bibitem{KP} T. Kato, G. Ponce, \textit{Commutator estimates and the Euler and Navier-Stokes equations}, Comm. Pure  Appl. Math.  XLI,  891--907, 1988.

\bibitem{KT} H. Koch and N. Tzvetkov, On the local well-posedness of the Benjamin-Ono
equation in $H^s(\R)$,   {\em Int. Math. Res. Not. IMRN} 2003, 1449--1464,
2003.

\bibitem{LLS}    D. Lannes, F.  Linares, and J-C. Saut,   The Cauchy problem for the Euler-Poisson system and derivation of the Zakharov-Kuznetsov equation, \textit{Prog. Nonlinear Diff. Eqs Appl.} 84, 181--213, 2013.


\bibitem{Latorre}
J. C. Latorre, A. A. Minzoni, N. F. Smyth and C.A. Vargas
\newblock Evolution of Benjamin-Ono solitons in
the presence of weak Zakharov-Kutznetsov lateral dispersion,
\newblock {\em Chaos} 16, 043103, 2006.






\bibitem{NahasPonce}
J.~Nahas, G.~Ponce,
\newblock On the persistent properties of solutions to semi-linear {S}chr\"odinger equation,
\newblock {\em Comm. Partial Differential Equations.} 34, 1--20, 2009.

\bibitem{Nascimento} A. C. Nascimento, On special regularity properties of the Benjamin-Ono-Zakharov-Kuznetsov (BO-ZK) equation,
{\em Commun. Pure Appl. Anal.}  19, 4285--4325, 2020.

\bibitem{Stein}
E. M. Stein,
\newblock The characterization of functions arising as potentials,
\newblock {\em Bull. Amer. Math. Soc.} 67, 102--104, 1961.





\bibitem{riano} O. G. Ria\~no, The IVP for a higher dimensional version of the
Benjamin-Ono equation in weighted Sobolev spaces, arXiv:1908.07079v1.

\bibitem{ribaud} F. Ribaud and S. Vento, Local and global results for the Benjamin-Ono-Zakharov-Kuznetsov
equation,  {\em Discrete Contin. Dyn. Syst.}  37, 449--483, 2017.

\bibitem{Schippa} R. Schippa, On the Cauchy problem for higher dimensional Benjamin-Ono and Zakharov-Kuznetsov equations, {\em Discrete Contin. Dyn. Syst.}  40, 5189--5215, 2020.

\bibitem{Taylor} M.E. Taylor, Tools for PDE, Pseudodifferential Operators, Paradifferential Operators, 
and Layer Potentials, Mathematical Surveys and Monographs, 81. American Mathematical Society, Providence, RI, 2000.



\bibitem{ZK}   V. E. Zakharov and E. A.   Kuznetsov, On three dimensional solitons, \textit{Sov. Phys. JETP.} 39,  285--286, 1974.

\end{thebibliography}
\end{document}